\documentclass[12pt]{amsart}
\usepackage[text={425pt,650pt},centering]{geometry}
\usepackage{graphicx}
\usepackage{epsfig}
\usepackage{amsmath}
\usepackage{amssymb}
\usepackage{color}
\usepackage{caption}
\usepackage{tikz}

\newcommand{\commentout}[1]{}

\def \Rset {{\mathbb R}}
\def \Cset {{\mathbb C}}
\def \Zset {{\mathbb Z}}

\def \Nset {{\mathbb N}}

\def \Tset {{\mathbb T}}

\newcommand{\nit}{\noindent}

\newcommand{\be}{\begin{equation}}
\newcommand{\ee}{\end{equation}}
\newcommand{\ba}{\begin{eqnarray}}
\newcommand{\ea}{\end{eqnarray}}
\newcommand{\bi}{\begin{itemize}}
\newcommand{\ei}{\end{itemize}}
\newcommand{\br}{\begin{eqnarray}}
\newcommand{\er}{\end{eqnarray}}

\newcommand{\ep}{\varepsilon}

\newtheorem{theo}{Theorem}[section]
\newtheorem{defin}{Definition}[section]

\newtheorem{lem}{Lemma}[section]

\newtheorem{rmk}{Remark}[section]
\newtheorem{quest}{Question}[section]

\begin{document}

\title[Inverse problems in homogenization]{Some Inverse Problems in Periodic Homogenization of  Hamilton-Jacobi Equations}

\author[S. Luo]{Songting Luo}
\address{Department of Mathematics\\
Iowa State University\\ Ames, IA 50011, USA} 
\email{luos@iastate.edu}

\author[H. V. Tran]{Hung V. Tran}
\address{Department of Mathematics\\
The University of Chicago\\ 5734 S. University Avenue Chicago, Illinois 60637, USA}
\email{hung@math.uchicago.edu}

\author[Y. Yu]{Yifeng Yu}
\address{Department of Mathematics\\
University of California at Irvine, California 92697, USA}
\email{yyu1@math.uci.edu}

\thanks{
The  work of SL is partially supported by NSF grant DMS-1418908,
the work of HT is partially supported by NSF grant DMS-1361236,
the work of YY is partially supported by NSF CAREER award \#1151919.
}

\date{}


\maketitle

\begin{abstract}
We look at the effective Hamiltonian $\overline H$ associated with the Hamiltonian $H(p,x)=H(p)+V(x)$
in the periodic homogenization theory.  
Our central goal is to understand the relation between $V$ and $\overline H$.  
We formulate some inverse problems concerning this relation.
Such  type of inverse problems  are in general very challenging.  
In the paper,  we discuss several special cases in both convex and nonconvex settings.  

\end{abstract}

\section{Introduction}  

\subsection{Setting of the inverse problem}

For each $\ep>0$, let $u^{\ep}\in C(\Rset^n\times [0,\infty))$ be the viscosity solution to the following Hamilton-Jacobi equation
\begin{equation}\label{HJ-ep}
\begin{cases}
u_t+ H\left(Du^{\ep},  {x\over \ep}\right)=0  \quad &\text{in  $\Rset^n\times (0,\infty)$},\\
u^{\ep}(x,0)=g(x) \quad &\text{on $\Rset^n$}.
\end{cases}
\end{equation}
The Hamiltonian $H=H(p,x)\in  C(\Rset^n\times \Rset^n)$ satisfies
\begin{itemize}
\item[(H1)] $x \mapsto H(p,x)$ is  $\Zset^n$-periodic, 

\item[(H2)] $p\mapsto H(p,x)$ is coercive uniformly in $x$,  i.e.,
$$
\lim_{|p|\to   +\infty}H(p,x)=+\infty   \quad \text{uniformly for $x\in  \Rset^n$},
$$
\end{itemize}
and the initial data $g\in \text{BUC}(\Rset^n)$, the set of bounded, uniformly continuous functions on $\Rset^n$.
 
It was proved by Lions, Papanicolaou and Varadhan \cite{LPV} that $u^{\ep}$, as $\ep\to 0$, converges locally uniformly to $u$, the solution of the  effective equation,
\begin{equation}\label{HJ-lim}
\begin{cases}
 u_t+\overline H(Du)=0  \quad  &\text{in  $\Rset^n\times (0,\infty)$},\\
 u(x,0)=g(x) \quad &\text{on $\Rset^n$}. 
\end{cases}
\end{equation}
The effective Hamiltonian $\overline H:\Rset^n \to  \Rset$ is determined by  the cell problems as follows.  
For any $p\in  \Rset^n$, we consider the following cell problem
\begin{equation}\label{C-p}
H(p+Dv, x)= c   \quad \text{in  $\Tset^n$},
\end{equation}
where $\Tset^n$ is the $n$-dimensional torus $\Rset^n/\Zset^n$.
We here seek for a pair of unknowns $(v,c) \in C(\Tset^n) \times \Rset$ in the viscosity sense.
It was established in \cite{LPV} that there exists a unique constant $c\in \Rset$ such that \eqref{C-p} has a solution $v\in C(\Tset^n)$.
We then denote by $\overline H(p)=c$.

In this paper, we always consider  the Hamiltonian $H$ of the form  $H(p,x)=H(p)+V(x)$. 
Our main goal is to study the relation between the potential energy $V$ and the effective Hamiltonian $\overline{H}$.
In the case where $H$ is uniformly convex, Concordel \cite{Co1, Co2} provided some first general results on
the properties of $\overline H$, which is convex in this case. 
In particular, she achieved some representation formulas of $\overline H$ by using optimal control theory
and showed that $\overline H$ has a flat part under some appropriate conditions on $V$.
The connection between properties of $\overline H$ and weak KAM theory can be found in E \cite{W-E},
Evans and Gomes \cite{EG}, Fathi \cite{F} and the references therein.
We refer the readers to Evans \cite[Section 5]{E-lecture} for a list of interesting viewpoints and open questions. 
 To date, deep properties of $\overline H$ are still not yet well understood.

In the case where $H$ is not convex, there have been not so many results on qualitative and quantitative properties of
$\overline H$.  Very recently, Armstrong, Tran and Yu \cite{ATY2013, ATY2014} studied nonconvex stochastic homogenization
and derived qualitative properties of $\overline H$
in the general  one dimensional case, and in some special cases in higher dimensional spaces.
The general case in higher dimensional spaces  is still out of reach.

We present here a different question concerning the relation between $V$ and $\overline H$.
In its simplest way, the question can be thought of as: how much can we recover the potential energy $V$
provided that we know $H$ and $\overline H$?
More precisely, we are interested in the following inverse type problem:

\begin{quest}
Let $H \in C(\Rset^n)$ be a given coercive function, and
$V_1, V_2 \in C(\Rset^n)$ be two given potential energy functions which are $\Zset^n$-periodic.
 Set $H_1(p,x)=H(p)+V_1(x)$ and $H_2(p,x)=H(p)+V_2(x)$ for $(p,x) \in \Rset^n \times \Rset^n$.
Suppose that $\overline H_1$ and $\overline H_2$ are two effective Hamiltonians corresponding to the two Hamiltonians  $H_1$ and $H_2$ respectively.   If 
$$
\overline H_1\equiv \overline H_2,
$$
then what can we conclude about the relations between $V_1$ and $V_2$? 
Especially,  can we identify some common ``computable"  properties shared by $V_1$ and $V_2$?
\end{quest}

To the best of our knowledge,  such kind of questions have never been explicitly stated and studied before.  
This is closely related to the exciting  projects of going beyond the well-posedness of the homogenization 
and understanding deep properties of the  effective Hamiltonian,  which are in general very hard. 
In this paper,  we discuss several special cases in high dimensional spaces 
and provide detailed analysis in one dimensional space for both convex and nonconvex $H$.  
Some first results for the viscous case are also studied.

\subsection{Main Results}

\subsubsection{Dimension $n\geq 1$}

\begin{theo}\label{m1} 
Assume $V_2 \equiv 0$.  Suppose that there exists $p_0\in \Rset^n$ such that $H\in C(\Rset^n)$ is differentiable at $p_0$ and $DH(p_0)$ is an irrational vector, i.e.,
\[
DH(p_0) \cdot m \neq 0 \quad \text{for all} \ m \in \Zset^n \setminus \{0\}.
\]
  Then
$$
\overline H_1(p_0)=\overline H_2(p_0) \quad  \mathrm{and} \quad \min_{\Rset^n} \overline H_1=\min_{\Rset^n} \overline H_2  \quad  \Rightarrow \quad  V_1\equiv 0.
$$
In particular, 
\be\label{ergo}
\overline H_1\equiv \overline H_2     \qquad   \Rightarrow  \qquad   V_1\equiv 0.
\ee
\end{theo}

Note that we do not assume $H$ is convex in the above theorem. 
As $V_2 \equiv 0$, it is clear that $\overline H_2=H$.
The theorem infers that if $\overline H_1(p_0)= H(p_0) $, $\min_{\Rset^n} \overline H_1=\min_{\Rset^n} H$ and $DH(p_0)$ is an irrational vector, then in fact $V_1 \equiv 0$.
The requirement on $DH(p_0)$ seems technical on the first hand, but it is, in fact, optimal.
If the set 
\[
G=\{DH(p)\,:\,\  \text{$H$ is differentiable at $p$ for $p\in \Rset^n$}\}
\]
 only contains rational vectors,  \eqref{ergo}  might fail.  See Remark \ref{rational}.  

If neither $V_1$ nor $V_2$ is constant, the situation usually  involves complicated dynamics and  becomes much harder to analyze.  In this paper, we establish some preliminary results.  A vector  $Q\in  \Rset^n$ satisfies a Diophantine condition if  there exist $C$, $\alpha>0$ such that 
\[
|Q\cdot k|\geq {C\over |k|^\alpha} \quad  \text{for any}\ k\in  \Zset^n \setminus \{0\}.  
\]

\begin{theo}\label{m2}  Assume that    $V_1, V_2 \in C^{\infty}(\Tset^n)$.  
\begin{itemize}

\item[(1)]   Suppose that $H\in C^2(\Rset^n)$,  $\sup_{\Rset^n}\|D^2H\|<+\infty$ and $H$ is superlinear.  Then for $i=1,2$ and any vector $Q\in  \Rset^n$ satisfying a Diophantine condition,
\be\label{difference1}
\int_{\Bbb T^n}V_i\,dx=\lim_{\lambda\to +\infty}\left( \overline H_{i}(\lambda P_{\lambda})-H(\lambda P_{\lambda})\right ).
\ee
Here $P_{\lambda}\in  \Rset^n$ is choosen such that $DH( \lambda P_{\lambda})=\lambda Q$.  In particular,   
$$
\overline H_1\equiv \overline H_2  \quad  \Rightarrow \quad \int_{\Bbb T^n}V_1\,dx= \int_{\Bbb T^n}V_2\,dx.
$$

\item[(2)] Suppose that $H(p)={1\over 2}|p|^2$.  We have that,  for $i=1,2$ and any irrational vector $Q\in  \Rset^n$, 
\be\label{difference}
\int_{\Bbb T^n}V_i\,dx=\lim_{\lambda\to +\infty}\left( \overline H_{i}(\lambda Q)-{1\over 2}\lambda^2|Q|^2\right )
\ee
and
\be\label{smallerror}
\lim_{\lambda\to +\infty}\left(\lambda^2|Q|^2-\max_{q\in  \partial \overline H_i(\lambda Q)}q\cdot \lambda Q \right)=0
\ee

\item[(3)] Suppose that $H(p)={1\over 2}|p|^2$.  If there exits $\tau>0$ such that
\be\label{strongdecay}
\sum_{k\in \Zset^n}(|\lambda_{k1}|^2+|\lambda_{k2}|^2)e^{|k|^{n+\tau}}<+\infty,  
\ee
then 
$$
\overline H_1\equiv \overline H_2  \quad  \Rightarrow \quad \int_{\Bbb T^n}|V_1|^2\,dx= \int_{\Bbb T^n}|V_2|^2\,dx.
$$
Here $\{\lambda_{ki}\}_{k\in  \Zset^n}$ are Fourier coefficients of $V_i$.  
\end{itemize}
\end{theo}

\begin{rmk} Due to the stability of the effective Hamiltonian,  (\ref{difference1}) and (\ref{difference}) still hold when  $V_1$, $V_2\in  C(\Bbb T^n)$.  
The equality \eqref{difference} is essentially known in case $Q$ satisfies a Diophantine condition. 
The average of the potential function is the constant term in the asymptotic expansion. See \cite{Arnold, D1, D2} for instance.

Moreover, when $H(p)={1\over 2}|p|^2$,  if $V_1$ and $V_2$ are both smooth,  
through direct computations of the asymptotic expansions, 
 $\overline H_1\equiv \overline H_2$  leads to a series of identical quantities
 associated with $V_1$ and $V_2$,   which involve complicated  combinations of  Fourier coefficients.
 It is very difficult to calculate those quantities and 
our goal is to to extract some new computable quantitites from those almost uncheckable ones.  
 The above theorem says that  the average and the  $L^2$ norm of the potential can be recovered.  
 See (\ref{formulaforL2}) for an explicit formula to compute the $L^2$ norm.  
 The fast decay condition (\ref{strongdecay}) is a bit restrictive at this moment.  
 It can be slightly relaxed if we transform the problem into the classical moment problem and apply Carleman's condition. 

 In fact, we conjecture that the distribution of the potential function 
should be determined by the effective Hamiltonian under reasonable assumptions. 
When $n=1$, this is proved in Theorem \ref{m3} for much more general Hamiltonians.  
High dimensions  will be studied in a future work.

\end{rmk}

\subsubsection{One dimensional case}

When $n=1$,  we have a much clearer understanding of this inverse  problem.  
Let us  first define some terminologies.

\begin{defin}
We say that   $V_1$ and $V_2$ have the same distribution if 
$$
\int_{0}^{1}f(V_1(x))\,dx=\int_{0}^{1}f(V_2(x))\,dx
$$
for any $f\in  C(\Rset)$.   
\end{defin}

\begin{defin}
$H:\Rset\to  \Rset$ is called  {\it strongly superlinear} if   there exists $a\in  \Rset$ such that 
the restriction of $H$ to $[a,+\infty)$ ($H|_{[a, +\infty)}:[a,+\infty) \to \Rset$) is smooth, strictly increasing, and 
\be\label{con}
\lim_{x\to +\infty}   {\psi^{(k)}(x)\over \psi^{(k-1)}(x)}=0   \quad \text{for all $k\in  \Nset$}.
\ee
Here $\psi=\psi^{(0)}=\left(H|_{[a,+\infty)}\right)^{-1}:[H(a), +\infty)\to  [a,  +\infty)$ 
and $\psi^{(k)}$ is the $k$-th derivative of $\psi$ for $k \in \Nset$.
\end{defin}

Note that condition \eqref{con} is only about the asymptotic behavior at $+\infty$. 
There is a large class of  functions satisfying the above condition, e.g. $H(p)=e^p$,  $H(p)=(c+|p|)^{\gamma}$
for $p \in [a,+\infty)$ for any $a\in \Rset$, $\gamma>1$ and $c\geq 0$.  
As nothing is required for the behavior of $H$ in $(-\infty,a)$ (except coercivity at $-\infty$), $H$ clearly can be nonconvex.

\begin{theo}\label{m3} Assume $n=1$ and $V_1$, $V_2\in  C(\Bbb T)$.  Then the followings hold:

\begin{itemize}
\item[(1)]  If $H$ is quasi-convex, then
$$
\text{$V_1$ and $V_2$ have the same distribution}\quad \Rightarrow   \quad \text{$\overline H_1 \equiv \overline H_2$}.
$$

\item[(2)] If  $H$ is strongly superlinear,  then
$$
\text{$\overline H_1\equiv \overline H_2$}  \quad \Rightarrow   \quad \text{$V_1$ and $V_2$ have the same distribution}.
$$
\end{itemize}
\end{theo}

When $H$ is nonconvex,   statement (1) in the above theorem is not true in general.  
In order to discuss about the general nonconvex situation, we first need some preparations.

Let us look at a basic nonconvex example of the Hamiltonians as following.   
This is a typical example of a nonconvex Hamiltonian with non-symmetric wells. 
Choose $F:[0,\infty)\to  \Rset$ to be a continuous function satisfying that (see Figure 1)

\begin{itemize}
\item[(i)]  there exist $0<\theta_3<\theta_2<\theta_1$ such that
\[
F(0)=0, \  F(\theta_2)={1\over 2}, \  F(\theta_1)=F(\theta_3)={1\over 3},
\]
and $\lim_{r\to +\infty}F(r)=+\infty$,

\item[(ii)] $F$ is strictly increasing on $[0,\theta_2]$ and $[\theta_1,+\infty)$, 
and $F$ is strictly decreasing on $[\theta_2,\theta_1]$.
\end{itemize}

\begin{center}
\includegraphics[scale=0.4]{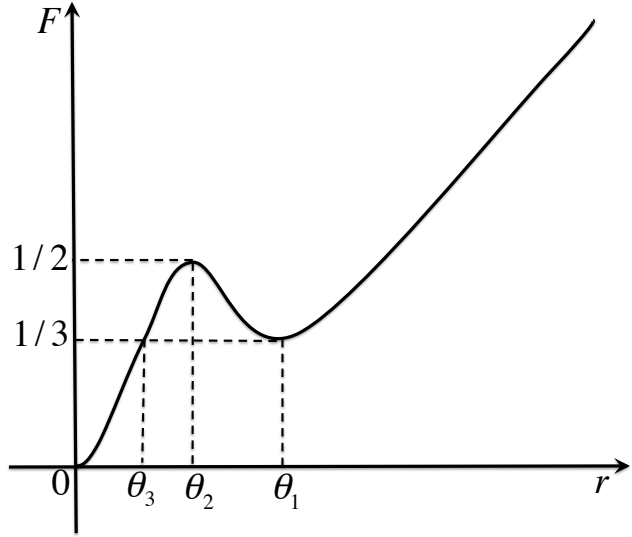}
\captionof{figure}{Graph of $F$}
\end{center}
The nonconvex Hamiltonian we will use intensively is $F(|p|)$.

 For $s\in (0,1)$, denote $V_s:[0,1] \to \Rset$ (see the left graph of Figure 2)
\be\label{family}
V_s(x)=
\begin{cases}
-{x\over s}  \quad &\text{for $x\in  [0,s]$},\\
{x-1\over 1-s}  \quad &\text{for $x\in  [s,1]$},
\end{cases}
\ee
and extend $V_s$ to $\Rset$ in a periodic way.
Let $H_s(p,x)=H(p)+V_s(x)$ for $(p,x)\in \Rset^n \times \Rset^n$. Denote by $\overline H_s$ the effective Hamiltonian corresponding to $H_s$.

\begin{defin}  We say that $V_1$, $V_2\in C(\Bbb T)$ are  macroscopically  indistinguishable if  
$$
\overline H_1\equiv  \overline H_2
$$
for any  coercive  continuous Hamiltonian $H:\Rset\to \Rset$. 
\end{defin}

Let $\hat V:[0,1] \to \Rset$ be  a piecewise linear function oscillating between 0 and -1 (see the right graph of Figure 2) such that

\begin{itemize}
\item there exist $0=a_1<c_1<a_2<\cdots<a_{m-1}<c_{m-1}<a_m=1$ for some $m\geq 2$ and
\[
\hat V(c_i)=-1 \quad \text{and} \quad  \hat V(a_i)=0,
\]

\item  $\hat V$ is linear within intervals $[a_i, c_i]$ and $[c_i, a_{i+1}]$ for $i=1,2,\ldots,m-1$.
\end{itemize}
Extend $\hat V$ to $\Rset$ in a periodic way.

\begin{theo}\label{m4} Let $s\in (0,1)$, $V_1=\hat V$, $V_2=V_s$.     

\begin{itemize}
\item[(1)]   $V_1$ and $V_2$ are macroscopically  indistinguishable if 
\be\label{balance}
\sum_{i=1}^{m-1}{(c_i-a_i)\over s}=\sum_{i=1}^{m-1}{(a_{i+1}-c_i)\over 1-s}.
\ee

\item[(2)]  For $H(p)=F(|p|)$,   
$$
\overline H_1\equiv  \overline H_2    \quad   \Rightarrow   \quad   \text{\eqref{balance} holds}.
$$
\end{itemize}
\end{theo}

\begin{rmk}   
From the above theorem,  we have that, for  $H(p)=F(|p|)$ and  $s,s'\in (0,1)$,   
$\overline H_{s'}\equiv \overline H_s$ if and only if $s=s'$.  
This demonstrates a subtle difference between convex and non-convex case.   
If $H:\Rset \to  \Rset$ is convex and even,  then the effective Hamiltonian $\overline H$ associated with $H(p)+V(x)$ for any $V\in C(\Bbb T^n)$  is also even.  
However,   this symmetry breaks down for the non-convex Hamiltonian $H(p)=F(|p|)$  since 
$$
\overline H_s(p)=\overline H_{1-s}(-p)\ne \overline H_s(-p)  \quad \text{if $s\ne {1\over 2}$}.
$$ 
\end{rmk}

\begin{center}
\includegraphics[scale=0.26]{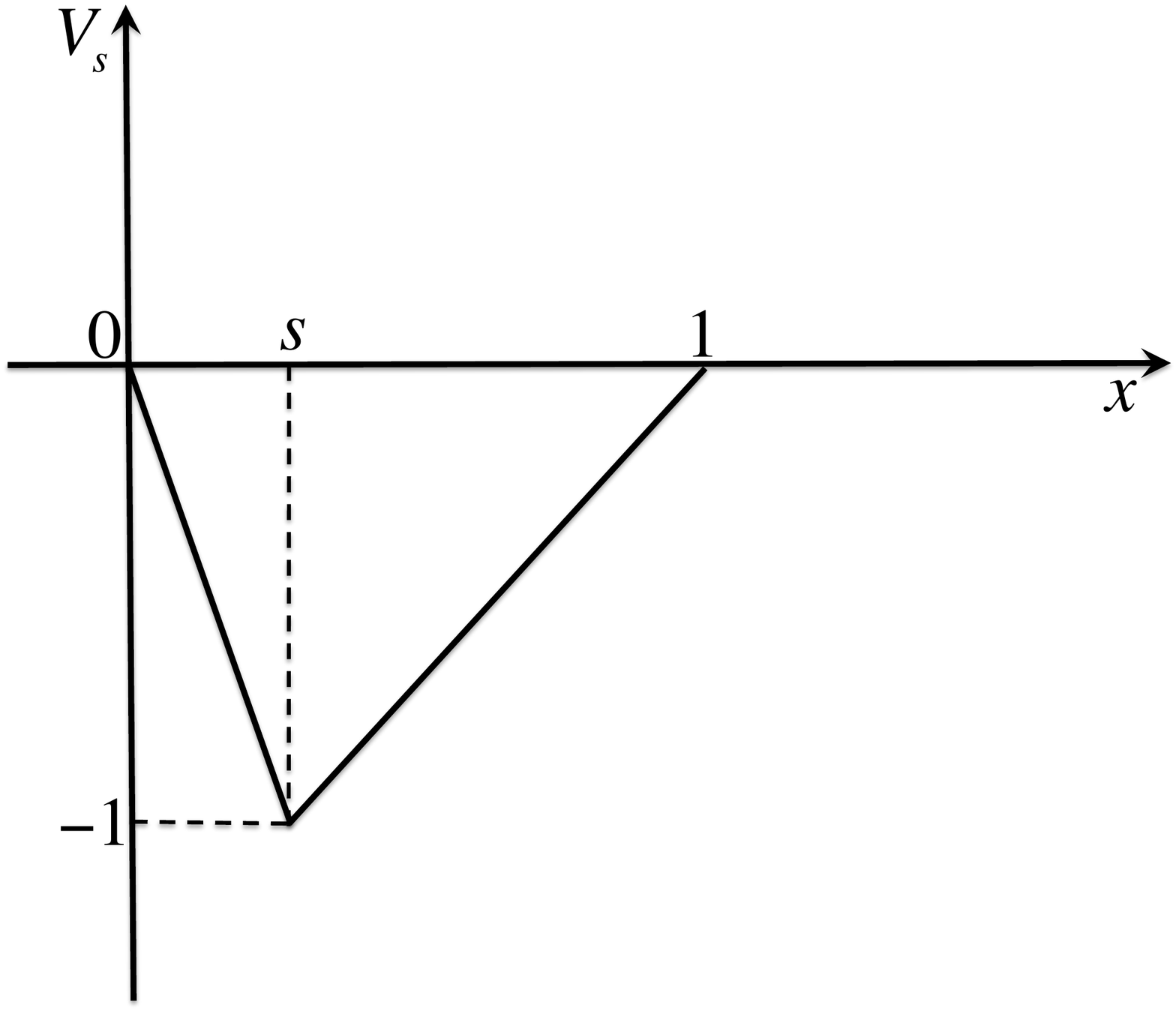}
\includegraphics[scale=0.26]{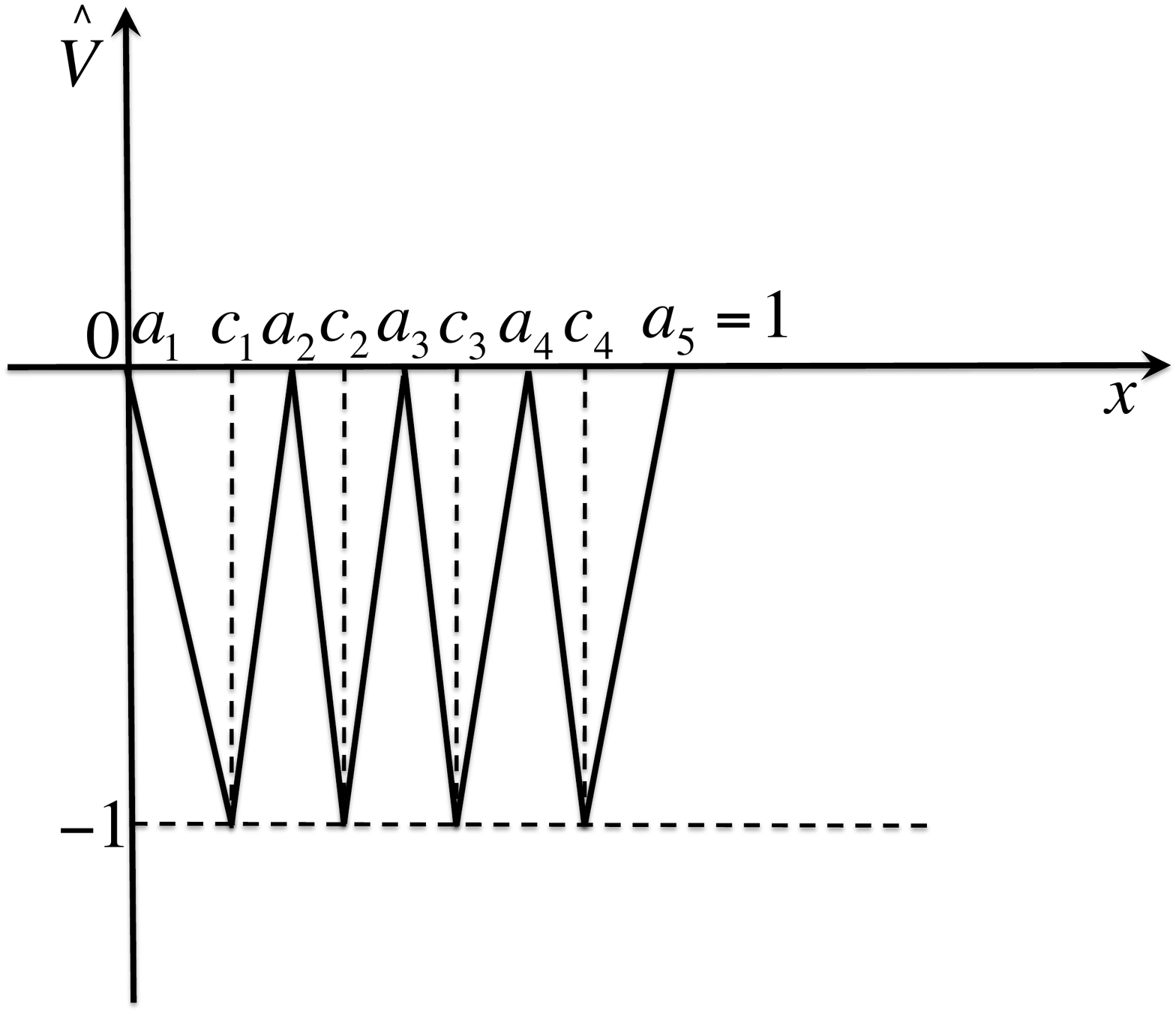}
\captionof{figure}{Left: Graph of $V_s$. \quad  Right:  Graph of $\hat V$  in case $m=5$.}
\end{center}

It is worth mentioning that \eqref{balance} is actually equivalent to the fact that
\begin{equation}\label{balance-n}
\sum_{i=1}^{m-1}(c_i-a_i)=s \quad \text{and} \quad \sum_{i=1}^{m-1}(a_{i+1}-c_i)=1-s.
\end{equation}
The relation \eqref{balance-n} says that the total length of the intervals where $\hat V$ is decreasing is $s$
and the total length of the intervals where $\hat V$ is increasing is $1-s$.
The assertion of Theorem \ref{m4} therefore means that, $\hat V$ and $V_s$ are macroscopically 
indistinguishable if and only if the total lengths of increasing of $\hat V$ and $V_s$ are the same
and the total lengths of decreasing of $\hat V$ and $V_s$ are the same.
In other words, the above means that the distribution of the increasing parts of $\hat V$ and $V_s$
are the same, and so are the decreasing parts.

We note that the requirement that $\hat V$ is piecewise linear is just for simplicity.
See Theorem \ref{m4-v1} for a more general result.

The requirement that $\hat V$ is oscillating between $0$ and $-1$, which is the same as $V_s$, is actually
much more crucial. If this is not guaranteed, then $\hat V$ and $V_s$ are not macroscopically 
indistinguishable in general. See Theorem \ref{m4-v2} for this interesting observation.

\subsubsection{Viscous Case}

We may also consider the same inverse problem for the viscous Hamilton-Jacobi equation.
For each $p\in \Rset^n$, the cell problem of interest is
\begin{equation}\label{viscous-p}
-d\Delta w+H(p+Dw)+V(x)=\overline H_d(p)   \quad \text{in  $\Tset^n$}
\end{equation}
for some given $d>0$. 
Due to the presence of the diffusion term,  
any detailed analysis of  the viscous effective Hamiltonian $\overline H_d(p)$ becomes considerably more difficult even in one dimensional space.  
In this paper,  we establish the following theorems which is a viscous analogue of Theorem \ref{m1}. 
\begin{theo}\label{m6}  Assume $V\in C^{\infty}(\Bbb T^n)$.
\begin{itemize}

\item[(1)]   Suppose that $H\in C^2(\Rset^n)$,  $\sup_{\Rset^n}\|D^2H\|<+\infty$ and $H$ is superlinear.  Then
$$
\overline H_d(p)\equiv H(p)  \quad \Rightarrow \quad  V\equiv 0.
$$

\item[(2)]  Suppose that $H(p)=|p|^2$.   If  $\overline H_d(p)=|p|^2+o(|p|^2)$    for $p$ in a neighborhood of the origin $O\in \Rset^n$,  then 
$V\equiv 0$. 

\end{itemize}
\end{theo}
When $n=1$ and $H(p)=|p|^2$,   the inverse problem is actually equivalent to the inverse problem 
associated with the  spectrum  of the Hill operator $L=-{d^2\over dx^2}-V$,  
which has been extensively studied in the literature.    
See the discussion in  Section 4 for details.  

\begin{theo}\label{m7} Assume that $V_1$, $V_2\in  C^{\infty}(\Bbb T^n)$.

\begin{itemize}

\item[(1)]  Suppose that $H\in C^2(\Rset^n)$,  $\sup_{\Rset^n}\|D^2H\|<+\infty$ and $H$ is superlinear.  Then for $i=1,2$ and any vector $Q\in  \Rset^n$ satisfying a Diophantine condition 
\be\label{difference-viscous}
\int_{\Bbb T^n}V_i\,dx=\lim_{\lambda\to +\infty}\left( \overline H_{i}(\lambda P_{\lambda})-H(\lambda P_{\lambda})\right ).
\ee
Here $P_{\lambda}\in  \Rset^n$ is choosen such that $DH( \lambda P_{\lambda})=\lambda Q$.  In particular,   
$$
\overline H_1\equiv \overline H_2  \quad  \Rightarrow \quad \int_{\Bbb T^n}V_1\,dx= \int_{\Bbb T^n}V_2\,dx.
$$

\item[(2)]  Let $H(p)=|p|^2$.   If  Fourier coefficients of $V_i$ satisfy  \eqref{strongdecay}, then
$$
\overline H_{d1}\equiv \overline H_{d2}   \quad \Rightarrow \quad  \int_{\Bbb T^n}|V_{1}|^{2}\,dx=\int_{\Bbb T^n}|V_{2}|^{2}\,dx.  
$$

\item[(3)]  If $n=1$ and $H(p)=|p|^2$, then
$$
\overline H_{d1}\equiv \overline H_{d2}   \quad \Rightarrow \quad  \begin{cases} \int_{\Bbb T^1}V_1\,dx=\int_{\Bbb T^1}V_2\,dx\\[3mm]\int_{\Bbb T^1}|V_{1}|^{2}\,dx=\int_{\Bbb T^1}|V_{2}|^{2}\,dx.  \end{cases}
$$

\end{itemize}

\end{theo}

\begin{rmk} Part {\rm(3)} in the above theorem is optimal.  We cannot expect to get other identical norms of $V_1$ and $V_2$. 
 See  Remark \ref{differenceviscous} for the connection  with the KdV equation.  This is different from the inviscid case. 

\end{rmk} 

\medskip

\nit {\bf Outline of the paper.}  In  Section  2,  we provide the proofs of Theorems \ref{m1} and \ref{m2}.    
Section 3 is devoted to the proofs  of Theorems  \ref{m3} and \ref{m4} and further detailed analysis in
the one dimensional setting.    
The connection between  the viscous case and the Hill operator will be discussed in Section 4.  
The proofs of Theorem \ref{m6} and \ref{m7}  are also given there.  

\medskip

\nit {\bf Acknowledgement.} We would like to thank Elena Kosygina for a suggestion in the viscous case and proposing the name ``macroscopically indistinguishable".  We are also grateful to Abel Klein  for very helpful discussions regarding  the Hill operator.

\section{Some results in the general dimensional case}

\begin{proof}[{\bf Proof of Theorem \ref{m1}}]  We first recall that as $V_2 \equiv 0$, we have $\overline H_2=H$ and therefore $\overline H_1=H$.

Since $\min_{\Rset^n}\overline H_1=\min_{\Rset^n} H+\max_{\Rset^n}V_1$,  we deduce  that $\max_{\Rset^n}V_1=0$.   
If $V_1$ is not constantly zero,  without loss of generality,  we may assume that for some $r>0$
\be\label{e1}
V_1(x)<0   \quad   \text{for $x\in B(0,r)$}.
\ee
Since $Q=DH(p_0)$ is an irrational vector,  there exists $T>0$ such that  for any $x\in  [0,1]^n$, there exists $t_x\in [0,T]$,  $z_x\in  \Zset^n$ such that
\be\label{e2}
|x-t_xQ-z_x|\leq {r\over 2}.
\ee
Let $u(x,t)$ be the viscosity solution to 
\begin{equation}\label{HJ-p0}
\begin{cases}
u_t+H(Du)+V_1(x)=0  \quad &\text{in  $\Rset^n\times (0,  +\infty)$},\\
u(x,0)=p_0\cdot x \quad &\text{on $\Rset^n$}.
\end{cases}
\end{equation}
Then $x \mapsto u(x,t)-p_0\cdot x$ is $\Zset^n$-periodic for each $t\geq 0$.  Owing to \eqref{e1}, \eqref{e2} and the following Lemma \ref{contact}, we have that 
$$
u(x,T)>p_0\cdot x-H(p_0)T  \quad \text{for  all $x\in \Rset^n$}.
$$
The continuity and periodicity of $x \mapsto u(x,T)-p_0\cdot x$ then yield that 
\begin{equation}\label{dist-delta}
\delta=\min_{x\in \Rset^n}\{u(x,T)-p_0\cdot x+H(p_0)T\}>0.
\end{equation}
Denote  
$$
u_1(x,t)=u(x,t+T)+H(p_0)T-\delta \quad \text{for $(x,t) \in \Rset^n \times [0,+\infty)$}.
$$
In light of \eqref{dist-delta}, $u_1(x,0)\geq p_0\cdot x=u(x,0)$ for all $x\in \Rset^n$. The usual  comparison principle implies that $u_1\geq u$.  Hence
$$
\min_{x\in \Rset^n}\{u_1(x,T)-p_0\cdot x+H(p_0)T\}\geq \delta.
$$
Now for $m\geq 2$,   define 
$$
u_m(x,t)=u_{m-1}(x, t+T)+ H(p_0)T-\delta  \quad \text{for $(x,t) \in \Rset^n \times [0,+\infty)$}.
$$
By using a similar argument and induction,  we deduce that  $u_m\geq u$ and 
$$
\min_{x\in \Rset^n}\{u_m(x,T)-p_0\cdot x+H(p_0)T\}\geq \delta   \quad \text{for all $m\in  \Nset$}.  
$$
Accordingly, for all $x\in \Rset^n$ and $m\in \Nset$,
$$
u(x,mT)\geq p_0\cdot x-H(p_0)mT+m\delta.
$$
Therefore,
$$
H(p_0)=\overline H_1(p_0)=\lim_{m\to \infty}-{u(0,mT)\over mT}\leq H(p_0)-{\delta\over T},
$$
which is absurd.  So $V_1\equiv 0$. 
\end{proof}

\begin{lem}\label{contact}
Suppose that $V_1\leq 0$ and $u(x,t)$ is the viscosity solution of \eqref{HJ-p0}, which is
$$
\begin{cases}
u_t+H(Du)+V_1(x)=0  \quad &\text{in  $\Rset^n\times (0,  +\infty)$},\\
u(x,0)=p_0\cdot x \quad &\text{on $\Rset^n$}.
\end{cases}
$$  
Suppose that  $H$ is differentiable at $p_0$ and there exists a point $(x_0,t_0)\in \Rset^n\times (0,  +\infty)$ such that
\[
u(x_0,t_0)=p_0\cdot x_0-H(p_0)t_0.
\]
Then 
$$
 V_1\left(x_0-(t_0-s)DH(p_0)\right)=0  \quad \text{for all $s\in  [0,t_0]$}. 
$$
\end{lem}

\begin{proof}
 Choose a convex and superlinear Hamiltonian $\widetilde H:\Rset^n\to  \Rset$ such that  
 $\widetilde H\geq H$,    $\widetilde H$ is differentiable at $p_0$, 
$$
\widetilde H(p_0)=H(p_0)   \qquad  \mathrm{and} \qquad      D\widetilde H(p_0)=DH(p_0).
$$
Let $\widetilde u:\Rset^n\times [0,  +\infty)\to \Rset$ be the unique solution to
$$
\begin{cases}
\widetilde u_t+\widetilde H(D\widetilde u)+V_1(x)=0   \quad &\text{in  $\Rset^n\times (0,  +\infty)$},\\
\widetilde u(x,0)=p_0\cdot x \quad &\text{on $\Rset^n$}.
\end{cases}
$$
By the comparison principle,  we have that 
$$
u\geq \widetilde u\geq p_0\cdot x-\widetilde H(p_0)t =  p_0\cdot x- H(p_0)t    \quad \text{in  $\Rset^n\times (0,  +\infty)$}.
$$
Hence
\begin{equation}\label{u-control}
\widetilde u(x_0,t_0)=p_0\cdot x_0-\widetilde H(p_0)t_0.
\end{equation}
The optimal control formula gives that
$$
\widetilde u(x_0,t_0)=\min_{\gamma\in \Gamma_{x_0,t_0}}\left\{p_0\cdot\gamma(0)+\int_{0}^{t_0}\left(L(\dot \gamma(s))-V_1(\gamma(s))\right)\,ds\right\}.
$$
Here $\Gamma_{x_0,t_0}$ is the collection of all absolutely continuous curves $\gamma$ 
such that $\gamma(t_0)=x_0$.  
Assume that   $\widetilde u(x_0,t_0)=p_0\cdot\xi(0)+\int_{0}^{t_0}L(\dot\xi(s))-V_1(\xi(s))\,ds$ 
for  some $\xi\in \Gamma_{x_0,t_0}$. 
 Since $L(\dot \xi(s))+\widetilde H(p_0)\geq p_0\cdot \dot \xi(s)$ for a.e. $s\in  [0,t_0]$,  we have that
\begin{align*}
p_0\cdot\xi(0)+\int_{0}^{t_0}L(\dot\xi(s))-V_1(\xi(s))\,ds
&\geq p_0\cdot x_0-\widetilde H(p_0)t_0-\int_{0}^{t_0}V_1(\xi(s))\,ds\\
&\geq p_0\cdot x_0-\widetilde H(p_0)t_0=p_0\cdot x_0- H(p_0)t_0.
\end{align*}
Accordingly,  $L(\dot \xi(s))+\widetilde H(p_0)= p_0\cdot \dot \xi(s)$ for a.e. $s\in  [0,t_0]$
 and $V_1(\xi(s))=0$ for $s\in  [0,t_0]$.  
So $\dot \xi(s)=D\widetilde H(p_0)=DH(p_0)$ for a.e. $s\in  [0,t_0]$. Thus
\[
V_1(\xi(s))= V_1\left(x_0-(t_0-s)DH(p_0)\right)=0  \quad \text{for all $s\in  [0,t_0]$}. 
\]

\end{proof}

\begin{rmk}\label{rational}  If the set 
\[
G=\{DH(p)\,:\,  \text{$H$ is differentiable at $p$ for $p\in \Rset^n$}\}
\]
 only contains rational vectors,   the conclusion of Theorem \ref{m1} might fail.   
Below is a simple example.

Let $n=2$. Suppose that  $V\in C^{\infty}(\Tset^2)$ and  $V\leq 0$.
Denote $Q=[0,1]^2$. We can think of $V$ as a function defined on $Q$ with periodic boundary condition.
Assume further that $\partial Q \subset \{V=0\}$.   Let
$$
H(p)=\max\{K_1(p_1),\  K_2(p_2)\} \quad \text{for all $p=(p_1,p_2)\in  \Rset^2$}.
$$
Here  $K_i\in C(\Rset)$ is coercive for $i=1,2$.   Then it is not hard to verify that
$$
\overline H(p)= H(p)=\max\{K_1(p_1),\  K_2(p_2)\} \quad \text{for all $p\in  \Rset^2$}.
$$
\end{rmk}

\medskip

\begin{proof}[{\bf  Proof of Theorem \ref{m2} (Part 1)}]   We first prove \eqref{difference1} for $i=1$. 

 Since $Q$ satifies a Diophantine condition,  there exists a unique smooth periodic  solution $v$ (up to an additive constant)  to
$$
 Q\cdot Dv=a_1-V_1 \quad \text{in} \ \Tset^n,
$$
for  $a_1=\int_{\Bbb T^n}V_1\,dx$.  Then it is easy to see that for $v_{\lambda}={v\over \lambda}$, 
$$
H(\lambda P_\lambda+Dv_{\lambda})+V(x)=H(\lambda P_{\lambda})+a_1+O\left({1\over \lambda^2}\right) \quad \text{in  $\Tset^n$}.
$$
Let $w_{\lambda}\in C(\Bbb T^n)$ be a viscosity solution to 
$$
H(\lambda P_\lambda+Dw_{\lambda})+V(x)=\overline H_1(\lambda P_{\lambda})  \quad \text{in  $\Rset^n$}.
$$
By looking at the places where $w_{\lambda}-v_{\lambda}$ attains its maximum and minimum,  we get that
$$
\overline H_1(\lambda P_{\lambda})=H(\lambda P_{\lambda})+a_1+O\left({1\over \lambda^2}\right),
$$
which yields \eqref{difference1}.
\end{proof}

\medskip

\begin{lem}\label{mono}  Assume that $V\in  C^{\infty}(\Bbb T^n)$.  Let $\overline H$ be the effective  Hamiltonian associated with ${1\over 2}|p|^2+V$.  Then
$$
|p|^2\geq  \max_{\{Q\in  \partial \overline H(p)\}}p\cdot Q.
$$
\end{lem}

\begin{proof}
If suffices to verify the above inequality at $p_0\in \Rset^n$, which is a differentiable point of $\overline H$.  
Let $w\in C^{0,1}(\Bbb T^n)$ be a viscosity solution to  the cell problem
$$
{1\over 2}|p_0+Dw|^2+V(x)=\overline H(p_0)   \quad \text{in $\Tset^n$}.
$$
Choose $\mu$ to be a Mather measure associated with the Hamiltonian ${1\over 2}|p_0+p|^2+V$.   
Mather measures are probability Borel measures on $\Rset^n\times \Bbb T^n=\{(q,x)|\  q\in  \Rset^n,\  x\in \Bbb T^n\}$
 which minimize the Lagrangian action among Euler-Lagrange flow invariant  probability Borel measures. 
 See \cite{EG, F} for the precise definition and relevant properties.
 Denote by $\sigma$  the projection of $\mu $ to the base space $\Bbb T^n$.   
Then  $w$ is $C^{1,1}$ on spt($\sigma$),  
$$
{1\over 2}|p_0+Dw|^2+V(x)=\overline H(p_0)     \quad \text{on  spt($\sigma$)}.
$$
Moreover,  
$$
p_0+Dw(x)=q   \quad \text{on spt$(\mu)$},   \qquad  \int_{\Rset^n\times \Bbb T^n}q\,d\mu=D\overline H(p_0),
$$
and
$$
\int_{\Rset^n\times \Bbb T^n}q\cdot D\phi(x)\,d\mu=0  \quad \text{for any $\phi\in C^1(\Bbb T^n)$}.
$$
Therefore
$$
 \int_{\Rset^n\times \Bbb T^n}|q|^2\,d\mu=p_0\cdot D\overline H(p_0)
$$
and
$$
|p_0|^2=\int_{\Rset^n\times \Bbb T^n}|q-Dw|^2\,d\mu
=p_0\cdot D\overline H(p_0)+\int_{\Rset^n\times \Bbb T^n}|Dw|^2\,d\mu.
$$
\end{proof}

\begin{proof}[{\bf  Proof of Theorem \ref{m2} (Part 2)}]  
Next we prove (\ref{difference}).   
Since $Q$ is just irrational,  the asymptotic expansion method is no longer applicable. 
 The proof becomes more involved and relies on the special structure of  the quadratic Hamiltonian.  

\nit {\bf Step 1:}   We claim that for any  $\lambda>0$,   
there exists $\tilde Q_{\lambda}\in \partial \overline H_1(\lambda Q)$ such that
\be\label{gradient}
\lim_{\lambda\to +\infty}\left(\overline H_1(\lambda Q)-{1\over 2}\tilde Q_{\lambda}\cdot \lambda Q\right)=\int_{\Bbb T^n}V_1\,dx.
\ee
It suffices to prove the above claim for any sequence $\{\lambda_m\}$ converging to $+\infty$.  
 Without loss of generality,  we consider the sequence $\{\lambda_m\}$ such that
$\lambda_m=m$ for all $m\in \Nset$.  For $m\geq 1$,  let  
\[
H_m(p,x)={1\over 2}|Q+p|^2+{1\over m^2}V_1(x) \quad \text{for all $(p,x) \in \Rset^n \times \Rset^n$},
\]
and denote by $\overline H_m$ its corresponding effective Hamiltonian.
Let $w_m \in C(\Tset^n)$ be a solution to the following cell problem
\begin{equation}\label{C-m}
{1\over 2}|Q+Dw_m|^2+{1\over m^{2}}V_1(x)=\overline H_m(Q)  \quad \text{in $\Tset^n$}.
\end{equation}
By a simple scaling argument, we can easily check that
\[
\overline H_m(Q)={1\over m^{2}}{\overline H_1 (mQ)}.
\]
Choose $\mu_m$ to be a Mather measure associated with  $H_m$.  Denote $\sigma_m$ as  the projection of $\mu_m$ to the base space $\Bbb T^n$.   Then  $w_m$ is $C^{1,1}$ on spt($\sigma_m$),  
$$
{1\over 2}|Q+Dw_m|^2+{1\over m^2}V_1(x)=\overline H_m(Q)     \quad \text{on  spt($\sigma_m$)}.
$$
Moreover,  
$$
Q+Dw_m(x)=q   \quad \text{on spt$(\mu_m)$},   \qquad  \int_{\Rset^n\times \Bbb T^n}q\,d\mu_m=Q_m\in \partial \overline H_m(Q)
$$
and
$$
\int_{\Rset^n\times \Bbb T^n}q\cdot D\phi(x)\,d\mu_m=0  \quad \text{for any $\phi\in C^1(\Bbb T^n)$}.
$$
Accordingly,
$$
\overline H_m(Q)=\int_{\Rset^n\times \Bbb T^n}{1\over 2}|q|^2+{1\over m^2}V_1(x)\,d\mu_m    \quad \mathrm{and} \quad   \int_{\Rset^n\times \Bbb T^n}{1\over 2}|q|^2\,d\mu_m={1\over 2}Q\cdot Q_m.
$$
Therefore
$$
\int_{\Rset^n\times \Bbb T^n}V_1(x)\,d\mu_m=m^2\left(\overline H_m(Q)-{1\over 2}Q\cdot Q_m\right)=\overline H_1(mQ)-{1\over 2}mQ\cdot mQ_m.
$$
Upon passing a subsequence if necessary,  we may assume that 
$$
\mu_m\rightharpoonup  \mu    \quad \text{weakly in  $\Rset^n\times \Bbb T^n$},
$$
for some probability measure $\mu$ in $\Rset^n \times \Tset^n$.
Let  $\sigma$ be the projection of $\mu$ to the base space $\Bbb T^n$.  
Owing  to the following  Lemma \ref{nonarnold},   we have that 
$$
Q=q  \quad \text{on spt$(\mu)$}  \quad \mathrm{and} \quad \int_{\Bbb T^n}Q\cdot D\phi(x)\,d\sigma=0  \quad \text{for any $\phi\in C^1(\Bbb T^n)$}.
$$
 Hence
$$
\int_{\Bbb T^n}\phi(x)\,d\sigma={1\over T}\int_{\Bbb T^n}\int_{0}^{T}\phi(x+tQ)\,dtd\sigma.
$$
Sending $T\to +\infty$,  since $Q$ is a irrational vector,   we have that $\sigma$ is actually the Lebesgue measure, i.e.,
$$
\int_{\Bbb T^n}\phi(x)\,d\sigma=\int_{\Bbb T^n}\phi(x)\,dx \quad \text{for any $\phi\in C^1(\Bbb T^n)$}.
$$
So 
\[
\int_{\Bbb T^n}V_1(x)\,dx=\lim_{m\to +\infty}\int_{\Rset^n\times \Bbb T^n}V_1(x)\,d\mu_m=\lim_{m\to +\infty}\left(\overline H_1(mQ)-{1\over 2}mQ\cdot mQ_m\right). 
\]
Since $\tilde Q_m=mQ_m\in  \partial \overline H_1(mQ)$,   our claim \eqref{gradient} holds. 

\medskip

\nit {\bf Step 2:}  Write
$$
f(\lambda)={1\over 2}\lambda ^2|Q|^2-\overline H_1(\lambda Q) \quad \text{for} \ \lambda \in (0,\infty).
$$
Then $f$ is locally Lipschitz continuous.  Owing to Lemma \ref{mono},   $f'\geq 0$ a.e.   Also  $f$ is uniformly bounded since 
$$
0\leq   {1\over 2}|p|^2-\overline H_1(p)\leq  -\int_{\Bbb T^n}V_1\,dx.
$$
Hence $\lim_{\lambda\to +\infty}f(\lambda)$ exists and  $\liminf_{\lambda\to +\infty}\lambda f'(\lambda)=0$.
So there exists a sequence $\lambda_m\to +\infty$ as $m\to +\infty$ such that $f$ is differentiable at $\lambda_m$ and 
$$
\lim_{m\to +\infty}\lambda_m f'(\lambda_m)=0.
$$
Note that if $f$ is differentiable at $\lambda$, then
$$
f'(\lambda)=\lambda |Q|^2-  Q\cdot q
$$
for any $q\in  \partial \overline H_1(\lambda Q)$.    Accordingly,  
$$
f'(\lambda_m)=\lambda_m|Q|^2-Q\cdot \tilde Q_{\lambda_m}
$$
for $\tilde Q_{\lambda_m}\in   \partial \overline H_1(\lambda_mQ)$ from \eqref{gradient}.   
Thus
$$
\lim_{m\to +\infty}(\lambda_{m}^{2}|Q|^2-\lambda_mQ\cdot \tilde Q_{\lambda_m})=0.
$$
Together with \eqref{gradient},  we have that 
$$
\lim_{\lambda\to +\infty}f(\lambda)=\lim_{m\to +\infty}f(\lambda_m)=-\int_{\Bbb T^n}V_1\,dx.
$$
Combining this with Lemma \ref{mono} and \eqref{gradient},  we also obtain  \eqref{smallerror}. 
\end{proof}

\begin{lem}\label{nonarnold}  Assume $V\in C^{\infty}(\Bbb T^n)$.  For $\ep>0$,  let $v_{\ep}\in C(\Tset^n)$ be a viscosity solution to 
$$
{1\over 2} |P+Dv_{\ep}|^2+\ep V(x)=\overline H_{\ep}(P) \quad \text{in $\Tset^n$}.
$$
Then
\begin{equation}\label{R-ep}
\lim_{\ep\to 0}\sup_{x\in \mathcal {R}_{\ep}}|Dv_{\ep}(x)|=0.
\end{equation}
Here $\mathcal {R}_{\ep}=\{x\in  \Tset^n \,:\,   \text{$v_{\ep}$ is differentiable at $x$}\}$. 
\end{lem}

\begin{proof}   We argue by contradiction.  
If \eqref{R-ep} were false,  then there would exist $\delta>0$ such that,  by passing to a subsequence if necessary,  
$$
\lim_{\ep\to 0}|Dv_{\ep}(x_\ep)|\geq \delta
$$
and $v_{\ep}(x_\ep)=0$ for some $x_{\ep}\in  \mathcal {R}_{\ep}\cap [0,1]^n$.   Let $\xi_\ep:(-\infty,  0]\to  \Rset^n$ be the backward characteristic associated with $P\cdot x+v_{\ep}$ with $\xi_{\ep}(0)=x_{\ep}$.  Then
$$
\ddot \xi_{\ep}=-\ep DV(\xi_{\ep})\quad \text{and}  \quad  \dot \xi_{\ep}=P+Dv_{\ep}(\xi_{\ep}),
$$
and for any $t_2<t_1\leq 0$, 
\be\label{cal}
P\cdot \xi_\ep (t_1)+v_{\ep}(\xi_\ep (t_1))-P\cdot \xi_\ep (t_2)-v_{\ep}(\xi_\ep (t_2))=\int_{t_2}^{t_1}{1\over 2}|\dot \xi_\ep (s)|^2+\overline H_{\ep}(P)\,ds.
\ee
Upon a subsequence if necessary,  we assume that
$$
\lim_{\ep\to 0}\xi_{\ep}=\xi_0    \quad \text{uniformly in $C^1[-1,0]$}.
$$
It is clear that  $\dot \xi_0\equiv  \widetilde P=P+P_1$ for some $|P_1|\geq \delta$.    However, it is easy to see that 
$$
\lim_{\ep\to 0}\overline H_{\ep}(P)={1\over 2}|P|^2 \quad \mathrm{and}  \quad \lim_{\ep\to 0}v_{\ep}=0   \quad \text{uniformly in $\Rset^n$}.
$$
Owing to \eqref{cal},  we obtain that 
$$
P\cdot  \widetilde P\geq {1\over 2}|\widetilde P|^2+{1\over 2}|P|^2.
$$
So $P=\widetilde P$.   This is a contradiction. 
\end{proof} 

\begin{rmk} If $v_{\ep}\in  C^{\infty}(\Bbb T^n)$ (e.g. in the regime of classical KAM theory),  standard maximum principle arguments lead to 
$$
\max_{\Bbb T^n}|D^2v_{\ep}|\leq C \sqrt {\ep}.
$$
for a constant $C$ depending only on $V$ and the dimension $n$.   Hence 
$$
\max_{\Bbb T^n}|Dv_{\ep}|=O( \sqrt {\ep}).
$$
If $P$ is an irrational vector,   \eqref{difference} implies that $\int_{\Bbb T^n}|Dv_{\ep}|^2\,dx=o( {\ep})$.  Accordingly,  
$$
\max_{\Bbb T^n}|Dv_{\ep}|=o( \sqrt {\ep}).
$$
Higher order approximations for more general Hamiltonian can be found in \cite{D2}. 
\end{rmk}

\begin{proof}[{\bf  Proof of Theorem \ref{m2} (Part 3)}]  
We now show that  
\[
 \int_{\Bbb T^n}|V_{1}|^{2}\,dx=\int_{\Bbb T^n}|V_{2}|^{2}\,dx.
\]
under the decay assumption  (\ref{strongdecay}).  

 Let  $Q$ be a vector satisfying a Diophantine condition.  
For $i=1,2$, we can explicitly solve the following two equations  in $\Tset^n$ 
by computing Fourier coefficients
\be\label{eqgroup}
\begin{cases}
Q\cdot Dv_{i1}=a_{i1}-V_i,\\
Q\cdot Dv_{i2}=a_{i2}-{1\over 2}|Dv_{i1}|^2.
\end{cases}
\ee
Here $a_{i1}=\int_{\Bbb T^n}V_i\,dx$ and $a_{i2}=\frac{1}{2}\int_{\Bbb T^n}|Dv_{i1}|^2\,dx$.  
Then for $\ep>0$ and $i=1,2$,  $v_{i\ep}=\ep v_{i1}+  \ep ^2 v_{i2}$ satisfy
$$
{1\over 2}|Q+Dv_{i\ep}|^2+\ep V_i={1\over 2}|Q|^2+\ep a_{i1} + \ep^2 a_{i2}+O(\ep ^3).
$$
Suppose that $w_{i\ep}\in  C(\Bbb T^n)$ is a viscosity solution to 
$$
{1\over 2}|Q+Dw_{i\ep}|^2+\ep V_i=\overline H_{i\ep}(Q)  \quad \text{in $\Tset^n$}.
$$
Here $\overline H_{i\ep}$ is the effective Hamiltonian associated with ${1\over 2}|p|^2+\ep V_i$.  
Since $v_{i\ep}\in C^{\infty}(\Bbb T^n)$,  
by looking at places where $v_{i\ep}-w_{i\ep}$ attains its maximum and minimum, we derive that
$$
\overline H_{i\ep}(Q)={1\over 2}|Q|^2+\ep a_{i1} + \ep^2 a_{i2}+O(\ep ^3).
$$
Note that, for $i=1,2$, 
\[
 \overline H_{i\ep}(Q)={\ep}\overline H_i\left({Q\over \sqrt {\ep}}\right).
\]
Accordingly,  
$$
\overline H_1\equiv \overline H_2  \quad \Rightarrow  \begin{cases}   \int_{\Bbb T^n}V_1\,dx=\int_{\Bbb T^n}V_2\,dx\\
\int_{\Bbb T^n}|Dv_{11}|^2\,dx=\int_{\Bbb T^n}|D v_{21}|^2\,dx.  \end{cases}
$$
Recall that $\{\lambda_{k1}\}_{k\in  \Zset^n}$ and $\{\lambda_{k2}\}_{k\in  \Zset^n}$ 
are the Fourier coefficients of $V_1$ and $V_2$ respectively.   Since $V_1$, $V_2\in  C^{\infty}(\Bbb T^n)$,   $\lambda_{ki}$ decays faster than any power of $k$, i.e. for  $i=1,2$ and any $m\in  \Nset$, 
\be\label{poly}
\lim_{|k|\to +\infty}  {|\lambda_{ki}| \cdot |k|^m}=0.
\ee
Note that we do not need the strong decay condition (\ref{strongdecay}) at this point.  Then  for any $Q$ satisfying a Diophantine condition,  
$\int_{\Bbb T^n}|Dv_{11}|^2\,dx=\int_{\Bbb T^n}|D v_{21}|^2\,dx$ implies that
\be\label{eq1}
\sum_{0\not= k\in  \Zset ^n}{ |k|^2|\lambda_{k1}|^2\over |Q\cdot k|^2}=\sum_{0\not= k\in  \Zset ^n}{ |k|^2|\lambda_{k2}|^2\over |Q\cdot k|^2}.
\ee
Our goal is to verify that 
$$
\sum_{0\not= k\in  \Zset ^n}|\lambda_{k1}|^2=\sum_{0\not= k\in  \Zset ^n}|\lambda_{k2}|^2.
$$
For $\delta>0$ and $\tau'={\tau\over 2}$, denote
$$
\Omega_{\delta}=\left\{Q\in  B_1(0)\,:\,  |Q\cdot k|\geq {\delta\over |k|^{n-1+\tau'}}  \quad \text{for all $0\not=k\in  \Zset^n$}\right\}.
$$
Clearly,   $\Omega_{\delta}$  is decreasing with respect to $\delta$ and $\cup_{\delta>0}\Omega_{\delta}$ has full measure, i.e., $\left|\cup_{\delta>0}\Omega_{\delta}\right|=|B_1(0)|$.  Due to (\ref{poly}),  
we  can  take derivatives of the above equality  \eqref{eq1}. 
It  leads to that,   for any $m\in  \Nset$, $\delta>0$,
\be\label{everym}
\sum_{0\not= k\in  \Zset ^n}{ |k|^{2m}|\lambda_{k1}|^2\over |Q\cdot k|^{2m}}
=\sum_{0\not= k\in  \Zset ^n}{ |k|^{2m}|\lambda_{k2}|^2\over |Q\cdot k|^{2m}}  \quad \text{for a.e. $Q\in \Omega_{\delta} $}. 
\ee
Owing to (\ref{strongdecay}),  we have that 
$$
\sum_{0\not= k\in  \Zset ^n}(|\lambda_{k2}|^2+|\lambda_{k2}|^2)e^{|k|\over |Q\cdot k|}<+\infty  \quad \text{for any $Q\in \Omega_{\delta}$}.
$$
This, together with  (\ref{everym}), implies that
$$
\sum_{0\not= k\in  \Zset ^n}d_{k}\left(1-\cos \left({|k|\over |Q\cdot k|}\right)\right)=0  \quad  \text{for a.e. $Q\in \Omega_{\delta} $}. 
$$
for $d_k=|\lambda_{k2}|^2-|\lambda_{k2}|^2$.  Sending $\delta\to 0$,  we derive that
$$
\sum_{0\not= k\in  \Zset ^n}d_{k}\left(1-\cos \left({|k|\over |Q\cdot k|}\right)\right)=0 \quad \text{for a.e. $Q\in B_1(0) $}. 
$$
Taking integration of the above with respect to $Q$ in $B_1(0)$ leads to 
$$
\sum_{0\not= k\in  \Zset ^n}d_{k}=0.
$$
\end{proof}

\begin{rmk}   From the above proof,  we actually derive that for $i=1,2$,  
\be\label{formulaforL2}
\int_{\Bbb T^n}|V_{i}|^2\,dx={2\over c_0}\int_{B_1(0)}\sum_{m=1}^{\infty}(-1)^{m-1}{\Delta^{(m-1)}a_{i2}(Q)\over (2m)! (2m-1)!}\,dQ.
\ee
Here $c_0=\int_{B_1(0)}\left(1-\cos {1\over |x_1|}\right)\,dx$,  $\Delta^{(m)}$ represents the $m$-th Laplacian and $a_{i2}(Q)=\frac{1}{2}\int_{\Bbb T^n}|Dv_{i1}|^2\,dx$, i.e.,  the coefficient of $\ep^2$ in the asymptotic expansion.  Furthermore,  it is clear that the conclusion of Theorem \ref{m2} 
 only depends on the behavior of the effective Hamiltonian when $|p|$ is large.  
  In fact,  when $n\geq 2$,  for any $M>0$,  it is easy to  construct  
two different smooth periodic functions $V_1$ and $V_2$ with big bumps such that 
$V_1\geq V_2$,  $\max_{\Bbb T^n}V_1=\max_{\Bbb T^n}V_2=0$ and 
$$
\overline H_1(p)\equiv  \overline H_2(p)   \quad \text{when $\overline H_2(p)\leq M$}.
$$
See {\rm(9.7)} in \cite{Bangert} for instance.  
\end{rmk}

\section{Detailed analysis in the one dimensional case}

\subsection{Convex Case}   We first present a proof of Theorem \ref{m3}.   

\begin{proof}[{\bf Proof of Theorem \ref{m3}}]

The proof of part (1) is quite straightforward.    
By the coercivity of $H$ and the stability of $\overline H$,   
we may assume  that $H(0)=0=\min_{\Rset}H$,  $H$ is strictly increasing on $[0, +\infty)$ and is strictly decreasing on $(-\infty, 0]$.  
In light of this assumption, the formula of the effective Hamiltonian is given by
$$
\begin{cases}
\displaystyle p=\int_{0}^{1}H_{+}^{-1}\left(\overline H_i(p)-V_i(x)\right)\,dx   \quad &\text{for $p\geq p_{+}$},\\
\overline H(p)\equiv   0    \quad &\text{for $p\in  [p_{-},  p_{+}]$},\\
\displaystyle p=\int_{0}^{1}H_{-}^{-1}\left(\overline H_i(p)-V_i(x)\right)\,dx   \quad &\text{for $p\leq p_{-}$}.
\end{cases}
$$
Here  
\[
p_{\pm}=\int_{0}^{1}H_{\pm}^{-1}(-V_1(x))\,dx=\int_{0}^{1}H_{\pm}^{-1}(-V_2(x))\,dx,
\]
 and $H_{-}^{-1}$,  $H_{+}^{-1}$  are the inverses of $H$ on $[0,+\infty)$ and $(-\infty, 0]$ respectively.   
 Clearly,  $\overline H_1\equiv \overline H_2$. 

Let us prove the second part (2).   
Since $\min \overline H_i=\max_{\Rset}V_i$,   we may assume that $\max_{\Rset}V_i=0$ for $i=1,2$.   
Apparently, for $i=1,2$ and $c\geq H(a)$,  
$$
\max\left\{p\in  \Rset \,:\,  \overline H_i(p)=c\right\}=\int_{0}^{1}\psi(c-V_i(x))\,dx.
$$
Hence
$$
\int_{0}^{1}\psi (\lambda-V_1(x))\,dx=\int_{0}^{1}\psi (\lambda-V_2(x))\,dx  \quad \text{for all $\lambda\geq H(a)$}.
$$
Therefore
$$
\int_{-M}^{0}\psi (\lambda-t)\,dF_1(t)=\int_{-M}^{0}\psi (\lambda-t)\,dF_2(t) \quad \text{for all $\lambda\geq H(a)$}.
$$
Here $F_i(t)=\left|\{x\in  [0,1]\,:\,  V_i(x)\leq t\}\right|$ is the distribution function of $V_i$ for $i=1,2$,  
and $-M<\min  \{\min V_1,\ \min V_2\}$.  
Denote $G(t)=F_1(t)-F_2(t)\in BV[-M,0]$.  Then 
\[
\int_{-M}^{0}\psi(\lambda-t)dG(t)=0,
\] 
and   $G(0)=G(-M)=0$.  By integration by parts,  we derive that 
$$
\int_{-M}^{0}\psi'(\lambda-t)G\,dt=0.
$$
For $m\in  \Nset\cap [H(a)+1, +\infty)$,  choose $x_m\in I_m=[m,  M+m]$ such that 
$$
|\psi'(x_m)|=\max_{x\in  I_m}|\psi'(x)|.
$$
Then 
$$
\lim_{m\to +\infty}{\max_{x\in  I_m}|\psi' (x)-\psi'(x_m)|\over |\psi'(x_m)|}\leq \lim_{m\to +\infty}\max_{x\in I_m}{M |\psi'' (x)|\over |\psi'(x)|}=0.
$$
Hence
\be\label{vanish1}
\int_{-M}^{0}G(t)\,dt=\lim_{m\to +\infty}{1\over M}\int_{-M}^{0}{\psi'(x_m)-\psi'(m-t)\over \psi'(x_m)}G(t)\,dt=0.
\ee
Define $G_0=G$ and $G_k(t)=\int_{-M}^{t}G_{k-1}(s)\,ds$ for all $k\in \Nset$.   
Through integration by parts,  using  \eqref{con} and the similar approach to obtain the above \eqref{vanish1},   we derive that  for $k\geq 1$,  
$$
\int_{-M}^{0}\psi^{(k+1)}(\lambda-t) G_{k}(t)dt=0  \quad \text{for all $\lambda\geq H(a)+1$},
$$
and  $G_{k}(-M)=G_{k}(0)=0$, i.e.,  
$$
\int_{-M}^{0}G_{k-1}(t)\,dt=0.
$$
Via integration by parts,  it is easy to prove that the above equality leads to 
$$
\int_{-M}^{0}t^kG\,dt=k!(-1)^k\int_{-M}^{0}G_k\,dx=0  \quad \text{for all $k \in \Nset\cup \{0\}$}.
$$
Hence $G=0$ a.e.  in $[-M,0]$.
Thus $F_1(t)=F_2(t)$ a.e.  in $[-M,0]$.   
Since both $F_1$ and $F_2$ are right-hand continuous,  we have that 
$$
F_1\equiv  F_2.
$$
\end{proof}

\begin{rmk}   Clearly,  the second part in the above theorem is not true if $H$ is not strongly superlinear.
For example, for $H(p)=|p|$, then for $i=1,2$,
$$
\overline H_i(p)=
\begin{cases}
\displaystyle 0   \quad &\text{if $|p|\leq -\int_{0}^{1}V_i(x)\,dx$},\\
\displaystyle |p|+\int_{0}^{1}V_i(x)\,dx \quad &\text{if $|p|\geq -\int_{0}^{1}V_i(x)\,dx$}.
\end{cases}
$$
Therefore, in this case, $\overline H_1=\overline H_2$ if and only if
\[
\int_{0}^{1}V_1(x)\,dx=\int_{0}^{1}V_2(x)\,dx,
\]
which is clearly much weaker than assertion of {\rm (2)}.

In fact, it is not even true for some strictly convex  Hamiltonians.   For example,  let $\psi\geq 0$ be a smooth and strictly concave function satisfying that  $\psi(0)=0$ and 
$$
\psi(q)=\Psi(q)=|q|(1-e^{-|q|})  \quad \text{for $|q|\geq 3$}.  
$$
Then $H=\psi^{-1}$ is strictly convex and coercive.  Choose two smooth periodic functions $V_1$ and $V_2$ with different distributions such that  

\begin{itemize}
\item $\max_{\Rset}V_i=0$ and $\{V_i\geq -3\}=[{1\over 4},  {3\over 4}]$ for $i=1,2$;

\item $V_1=V_2$  on  $[{1\over 4},  {3\over 4}]$  and $\int_{0}^{1}V_1(x)\,dx=\int_{0}^{1}V_2(x)\,dx$;

\item Furthermore,
$$
\int_{0}^{1}V_1(x)e^{V_1(x)}\,dx=\int_{0}^{1}V_2(x)e^{V_2(x)}\,dx,\quad    \int_{0}^{1}e^{V_1(x)}\,dx=\int_{0}^{1}e^{V_2(x)}\,dx.
$$
\end{itemize}
Clearly,    for any  $c\geq 0$,
\begin{multline*}
\int_{0}^{1}\psi(c-V_1(x))\,dx-\int_{0}^{1}\psi(c-V_2(x))\,dx\\
=\int_{0}^{1}\Psi(c-V_1(x))\,dx-\int_{0}^{1}\Psi(c-V_2(x))\,dx=0.
\end{multline*}
Thus
$$
\overline H_1\equiv \overline H_2.
$$
\end{rmk}
\begin{rmk}  
When $n\geq 2$,  that $V_1$ and $V_2$ have the same distribution is  not  sufficient to yield that $\overline H_1 \equiv  \overline H_2$.  
Here is a simple example for $n=2$.  

Let $H(p)=|p|^2$. For $x=(x_1,x_2) \in \Rset^2$, set
\[
V_1(x)=-\sin  (2\pi x_{2}) \quad \text{and}  \quad V_2(x)=-\sin (2\pi x_{1}).
\]
Clearly,  $V_1$ and $V_2$ have the same distribution.  However,  $\overline H_1 \ne \overline H_2$.  In fact,  it is easy to see that 
$$
\overline H_1(p)=|p_1|^2+h(p_2)    \quad \mathrm{and}  \quad \overline H_2(p)=|p_2|^2+h(p_1).
$$
Here $h(t)$ is given by
$$
h(t)=
\begin{cases}
1    \quad &\text{if $|t|\leq  \int_{0}^{1}\sqrt {1+\sin (2\pi \theta)}\,d\theta$},\\
|t|=\int_{0}^{1}\sqrt { h(t)+\sin (2\pi\theta)}\,d\theta  \quad &\text{otherwise}.
\end{cases}
$$
In fact,  we should be able to identify precise necessary and sufficient conditions for $\overline H_1\equiv \overline H_2$ when $V_1$ and $V_2$ have finite frequencies.   This will be in a forthcoming  paper.  
\end{rmk}

\subsection{Nonconvex Case}

\begin{proof}[{\bf Proof of Theorem \ref{m4}}]
We proceed the proof in several steps.

\nit {\bf Step 1:}   For $s\in (0,1)$,  denote
$$
p_{+,s}=\max\{p\geq 0\,:\,\   \overline H_s(p)=0\}.
$$
Let 
\[
\begin{cases}
\psi_1=\left(F|_{[\theta_1,\infty)}\right)^{-1}:[{1\over 3}, +\infty)\to  [\theta_1, +\infty),\\
 \psi_2=\left(F|_{[\theta_2,\theta_1]}\right)^{-1}: [{1\over 3}, {1\over 2}]\to  [\theta_2, \theta_1],\\
 \psi_3=\left(F|_{[0,\theta_2]}\right)^{-1}:[0, {1\over 2}]\to  [0,\theta_2].
\end{cases}
\]
  Define  $f_s$ to be a periodic function satisfying that
$$
f_s(x)=
\begin{cases}
\psi_3(-V_s(x))   \quad &\text{for $x\in  [0, {s\over 3}]\cup  [{1\over 2}+{s\over 2},  1]$}\\
\psi_1(-V_s(x)) \quad &\text{for $x\in  ({s\over 3},  {1\over 2}+{s\over 2})$}.
\end{cases}
$$
It is easy to check that any solution $w'=f_s(x)$ is a viscosity solution to 
$$
F(|w'|)+V_s(x)=0  \quad \text{in  $\Rset$}. 
$$
We claim that 
\begin{align}\label{rightend}
p_{+,s}&=\int_{[0,1]}f_s(x)\,dx\\
&=\int_{1\over 3}^{1}\psi_1(y)\,dy+\int_{0}^{1\over 3}\psi_3(y)\,dy+(1-s)\int_{1\over 3}^{1\over 2}(\psi_3-\psi_1)(y)\,dy. \notag
\end{align}
In fact,  assume that $v\in  C^{0,1}(\Rset)$ is a periodic viscosity solution to 
$$
F(|p_{+,s}+v'|)+V_s(x)=0   \quad \text{in  $\Rset$}.
$$
Write $u=p_{+,s}x+v$.  Obviously,  
$$
u'(x)\leq f_s(x)    \quad \text{for a.e.   $x\in  [0,1]\backslash   \left[{1\over 2}+{s\over 2},  {2\over 3}+{s\over 3}\right]$}.
$$
Suppose that  there exists $x_0\in   ({1\over 2}+{s\over 2},  {2\over 3}+{s\over 3})$ such that $u'(x_0)$ exists and 
$$
u'(x_0)=\psi_j(-V_s(x_0))     \quad \text{for $j=1$ or $j=2$}.
$$
Since  $u'(({2\over 3}+{s\over 3},1))\subset (-\infty,  \theta_3]$,  due to Lemma \ref{mean},   we have that 
$$
F(\theta_2)+V_s(x_1)\leq 0  \quad \text{for some $x_1\in   \left[x_0,  {2\over 3}+{s\over 3}\right]$}.
$$
This is impossible since $-V_s<{1\over 2}$ in  $\left({1\over 2}+{s\over 2},  {2\over 3}+{s\over 3}\right]$.  Accordingly,  
$$
u'(x)\leq \psi_3(-V_s(x))   \quad \text{  in  $ \left({1\over 2}+{s\over 2},  {2\over 3}+{s\over 3}\right)$}.
$$
So  $u'(x)\leq   f_s(x)$  and hence the first equality of claim \eqref{rightend} holds.  An easy calculation leads to 
$$
\int_{[0,1]}f_s(x)\,dx=\int_{1\over 3}^{1}\psi_1(y)\,dy+\int_{0}^{1\over 3}\psi_3(y)\,dy+(1-s)\int_{1\over 3}^{1\over 2}(\psi_3-\psi_1)(y)\,dy,
$$
which implies the second equality of \eqref{rightend}. 

\medskip

\nit {\bf Step 2:}    We show that $\hat V$ and $V_s$ are macroscopically indistinguishable if \eqref{balance} holds.   Now we only assume that $H:\Rset\to \Rset$ is continuous and coercive.  Fix $s\in  (0,1)$, $1\leq i\leq m-1$, and  define $\tilde V\in  C(\Bbb T)$ as 
$$
\tilde V(x)=V_s\left ({x-a_i\over a_{i+1}-a_i}\right )   \quad \text{for $x\in   [a_i, a_{i+1}]$}.
$$
Let $\overline {\tilde H}$ be the effective Hamiltonian associated with $H(p)+\tilde V$.  Since $\tilde V$  is basically a piecewise rescaling of $V_s$,  $\overline {\tilde H}\equiv  \overline H_s$.   For fixed $p\in  \Rset$,  let $v$ be a  continuous periodic  viscosity solution to 
$$
H(p+v')+V_s(x)=\overline H_s(p)   \quad \text{in  $\Rset$}
$$
subject to $v(0)=v(1)=0$.     
Define, for $x\in [0,1]$, 
$$
\tilde v(x)=(a_{i+1}-a_i)v\left (x-a_i\over a_{i+1}-a_i\right)   \quad \text{for $x\in  [a_i, a_{i+1}]$, $1\leq i \leq m-1$},
$$
and extend $\tilde v$ to $\Rset$ in a periodic way.
It is easy to check that $\tilde v$ is a viscosity to 
$$
H(p+\tilde v')+\tilde V(x)=\overline H_s(p)  \quad \text{ in $\Rset$}.  
$$
Next we set $\tau:\Rset\to  \Rset$ as follows: when $x\in  [0, 1]$,  
$$
\tau (x)=
\begin{cases}
\displaystyle a_i+(a_{is}-a_i){x-a_i \over c_i-a_i} \quad &\text{for $x\in  [a_i,  c_i]$},\\
\displaystyle a_{i+1}+(a_{is}-a_{i+1}){x-a_{i+1}\over c_i-a_{i+1}}   \quad &\text{for $x\in  [c_i,  a_{i+1}]$},
\end{cases}
$$
for $a_{is}=(1-s)a_i+sa_{i+1}$. Then extend $\tau'$ periodically.   Owing to Lemma \ref{change-va},   $w'(x)=p+\tilde v'(\tau (x))$ is a viscosity solution to 
$$
H(w')+\hat V(x)=\overline H_s(p) \quad \text{in $\Rset$}.  
$$
Note that $w'$ is periodic and
\begin{align*}
w(1)-w(0)&=p+\int_{0}^{1}{\tilde v}'(\tau (x))\,dx\\[5mm]
&=p+v(s)\left(\sum_{i=1}^{m-1}{(c_i-a_i)\over s}-\sum_{i=1}^{m-1}{(a_{i+1}-c_i)\over 1-s}\right)\\[5mm]
&=p.
\end{align*}
 Hence $\overline H(p)=\overline H_s(p).$  Here $\overline H$ represents the effective Hamiltonian associated with $H(p)+\hat V$.

\medskip

\nit{\bf Step 3:}  Finally,  we prove that for $H=F(|p|)$,   $\overline H_1\equiv \overline H_2$  implies that \eqref{balance} holds. 

In fact,  assume that $\overline H_1\equiv \overline H_2=\overline H_s$ for some $s\in  (0,1)$.   
Clearly,   there exists a unique $s'\in   (0,1)$ such that  \eqref{balance} holds, i.e.,
\[
\sum_{i=1}^{m-1}{(c_i-a_i)\over s'}=\sum_{i=1}^{m-1}{(a_{i+1}-c_i)\over 1-s'}.
\]
Due to the Step 2,  $\hat V$ and $V_{s'}$ are macroscopically indistinguishable.  In particular, 
$$
\overline H_1\equiv \overline H_{s'}.
$$
This implies  $\overline H_s\equiv \overline H_{s'}$.    So $p_{+,s}=p_{+,s'}$.   By \eqref{rightend},  $s=s'$. 

\end{proof}


The following lemma was proved in Armstrong, Tran, Yu \cite{ATY2014}. We state it here as it is needed in the proof of the above theorem.
As its proof is simple, we also present it here for the sake of completeness.

\begin{lem}[Generalized mean value theorem]\label{lem:meanvalue} \label{mean}  
 Suppose that $u\in C([0,1],\Rset)$ and, for some $a,b \in \Rset$,
 \[
 u'(0^+)=\lim_{x \to 0^+} \frac{u(x)-u(0)}{x}=a \quad \text{and} \quad u'(1^-)=\lim_{x\to 1^-}\frac{u(1)-u(x)}{1-x}=b.
 \]
 Then the followings hold:

\noindent {\rm(i)} If $a<b$, then for any $c\in (a,b)$, there exists $x_c\in (0,1)$ such that $c \in D^- u(x_c)$, i.e., 
$$
u(x)\geq u(x_c)+c(x-x_c)-o(|x-x_c|) \quad \text{for} \ x\in (0,1).
$$
{\rm(ii)} If $a>b$, then for any $c\in (b,a)$, there exists $x_c\in (0,1)$ such that $c \in D^+ u(x_c)$, i.e., 
$$
u(x)\leq u(x_c)+c(x-x_c)+o(|x-x_c|)\quad \text{for} \ x\in (0,1).
$$
\end{lem}

\begin{proof}
We only prove (i). For $c\in (a,b)$, set $w(x):=u(x)-cx$ for $x\in [0,1]$. There exists $x_c\in [0,1]$ such that
\[
 w(x_c)=\min_{x\in [0,1]} w(x).
\]
Note that $x_c \neq 0$ and $x_c\neq 1$ as $c\in (a,b)$. Thus $x_c \in (0,1)$, which of course yields that $c \in D^- u(x_c)$.
\end{proof}

\begin{lem} \label{change-va}
 Suppose that $\tau:\Rset\to  \Rset$ is Lipschitz continuous and $\tau'>0$ a.e.     Assume that $u$ is a viscosity solution of 
$$
H(u')+V_1(x)=0  \quad \text{in  $\Rset$}.
$$
Let $V_2(x)=V_1(\tau(x))$ and $w'(x)=u'(\tau(x))$.  Then $w$ is a viscosity solution to 
$$
H(w')+V_2(x)=0   \quad \text{in  $\Rset$}.
$$
\end{lem}

\begin{proof}
The proof follows a straightforward change of variables.  We leave it as an exercise for the readers. 
\end{proof}

The following result is a more general version of Theorem \ref{m4}. Before stating the theorem, we first
give a definition of the potential energy $\hat V_1$.

Let $\hat V_1:[0,1] \to \Rset$ be a function oscillating between $0$ and $-1$ (see Figure 3) such that
\begin{itemize}
 \item there exist $0=a_1<c_1<a_2<\cdots<a_{m-1}<c_{m-1}<a_m=1$ for some $m\geq 2$ and
\[
 \hat V_1(c_i)=-1 \quad \text{and} \quad \hat V_1(a_i)=0.
\]
\item $\hat V_1$ is strictly decreasing on $[a_i,c_i]$ and is strictly increasing on $[c_i,a_{i+1}]$
for $i=1,2,\ldots, m-1$.
\end{itemize}
Extend $\hat V_1$ to $\Rset$ in a periodic way.

\begin{center}
\includegraphics[scale=0.4]{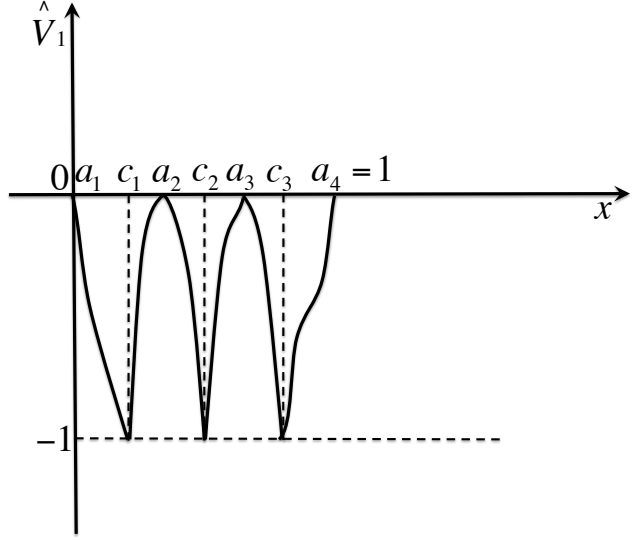}
\captionof{figure}{Graph of $\hat V_1$}
\end{center}

\begin{theo} \label{m4-v1}
Fix $s \in (0,1)$. The two potentials $V_s$ and $\hat V_1$ are macroscopically indistinguishable
if and only if, for all $f \in C(\Rset)$,
\begin{equation}\label{balance-g1}
 \sum_{i=1}^{m-1} \int_{a_i}^{c_i} f(\hat V_1(x))\,dx=\int_0^s f(V_s(x))\,dx=s \int_{-1}^0 f(y)\,dy,
\end{equation}
and
\begin{equation}\label{balance-g2}
 \sum_{i=1}^{m-1} \int_{c_i}^{a_{i+1}} f(\hat V_1(x))\,dx=\int_s^1 f(V_s(x))\,dx=(1-s) \int_{-1}^0 f(y)\,dy.
\end{equation}
\end{theo}
The proof of this theorem is basically the same as that of Theorem \ref{m4}. We hence omit it.
This result says that $V_s$ and $\hat V_1$ are macroscopically indistinguishable if and only if
the distribution of the increasing parts and decreasing parts are the same respectively.
Note that both $V_s$ and $\hat V_1$ are oscillating between $0$ and $-1$ here.

Let $\hat V_2:[0,1] \to \Rset$ be a function (see Figure 4) such that
\begin{itemize}
 \item $\hat V_2(0)=0$, $\hat V_2(\frac{1}{6})=-\frac{2}{5}$, $\hat V_2(\frac{1}{3})=0$,
$\hat V_2(\frac{11}{30})=-\frac{2}{5}$, $\hat V_2(\frac{2}{3})=-1$, $\hat V_2(\frac{29}{30})=-\frac{2}{5}$,
and $\hat V_2(1)=0$.

\item $\hat V_2$ is piecewise linear in the intervals $[0,\frac{1}{6}]$, $[\frac{1}{6},\frac{1}{3}]$,
$[\frac{1}{3},\frac{11}{30}]$, $[\frac{11}{30},\frac{2}{3}]$, $[\frac{2}{3},\frac{29}{30}]$, and $[\frac{29}{30},1]$.
\end{itemize}
\begin{center}
\includegraphics[scale=0.4]{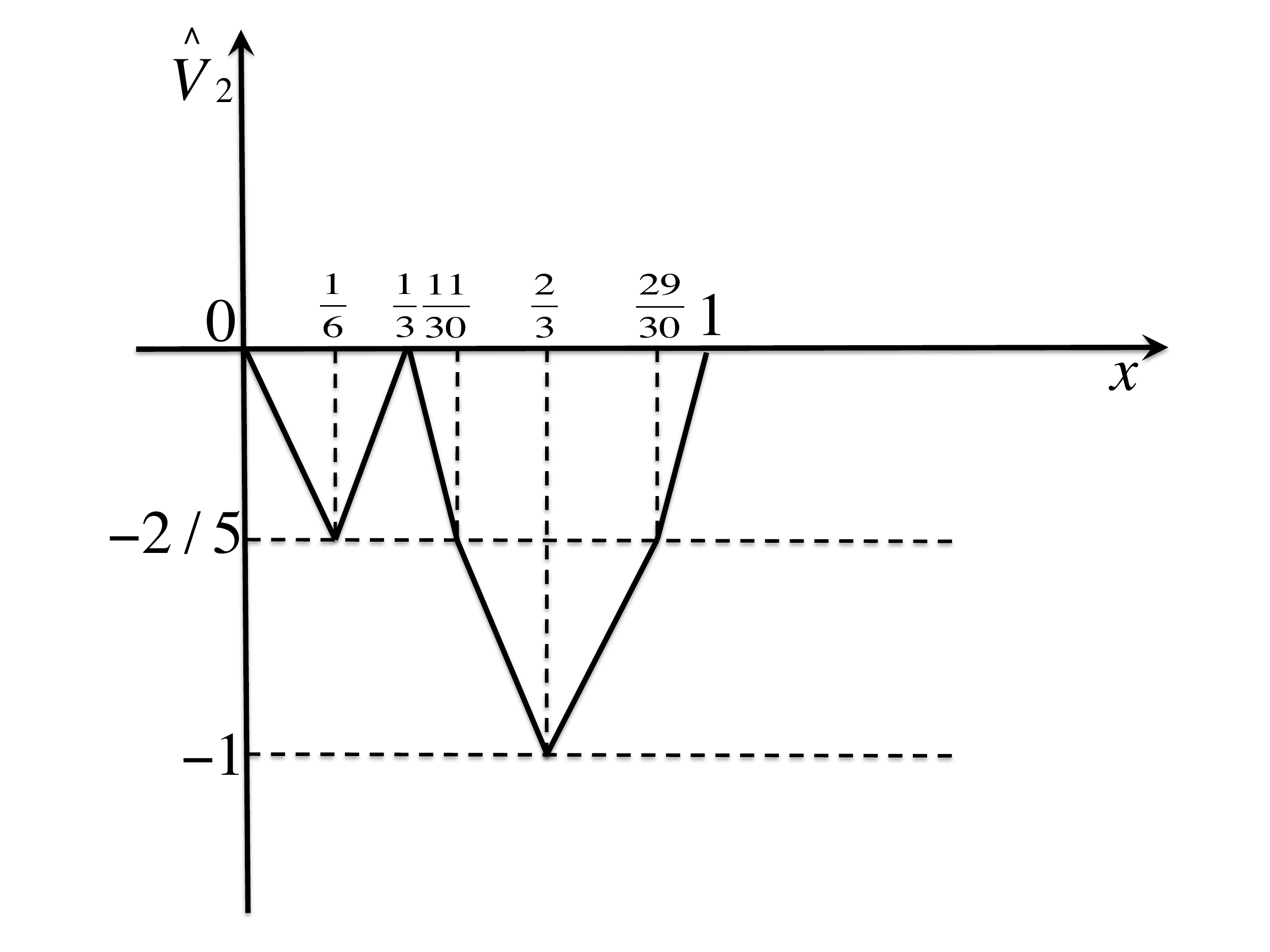}
\captionof{figure}{Graph of $\hat V_2$}
\end{center}

Extend $\hat V_2$ to $\Rset$ in a periodic way. 
Note that we choose $\hat V_2(\frac{1}{6})=-\frac{2}{5}$ for is just for simplicity 
so that we can use the nonconvex $F$ introduced earlier. 
It is clear that $V_{\frac{1}{2}}$ and $\hat V_2$ have the same distribution of the increasing parts as well
as the decreasing parts.

\begin{theo}\label{m4-v2}
The two potentials $V_{\frac{1}{2}}$ and $\hat V_2$ are  not macroscopically indistinguishable.
\end{theo}

\begin{proof}
Assume by contradiction that $V_s$ and $\hat V_2$ are macroscopically indistinguishable.

Set $H(p)=F(|p|)$ for $p\in \Rset$. Recall that
\[
 p_{+,\frac{1}{2}}=\max\{p\geq 0\,:\, \overline H_{\frac{1}{2}}(p)=0\},
\]
and in light of Step 1 in the proof of Theorem \ref{m4}, we have that
\begin{align*}
p_{+,\frac{1}{2}}&=\int_{[0,1]}f_{\frac{1}{2}}(x)\,dx\\
&=\int_{1\over 3}^{1}\psi_1(y)\,dy+\int_{0}^{1\over 3}\psi_3(y)\,dy+\frac{1}{2}\int_{1\over 3}^{1\over 2}(\psi_3-\psi_1)(y)\,dy\\
&=\int_{0}^{1\over 3}\psi_3(y)\,dy+\frac{1}{2}\int_{1\over 3}^{1\over 2}\psi_3(y)\,dy+
\frac{1}{2}\int_{1\over 3}^{1\over 2}\psi_1(y)\,dy+\int_{1\over 2}^{1}\psi_1(y)\,dy.
\end{align*}

Let $\overline H_2$ be the effective Hamiltonian associated with $H(p)+\hat V_2(x)$, and
\[
 p_{+,2}=\max \{ p \geq 0\,:\, \overline H_2(p)=0\}.
\]
The same method as Step 1 in the proof of Theorem \ref{m4} can be used to give that
\begin{align*}
 p_{+,2}&=\int_0^{\frac{13}{36}} \psi_3(-\hat V_2(x))\,dx 
+\int_{\frac{13}{36}}^{\frac{11}{12}} \psi_1(-\hat V_2(x))\,dx+\int_{\frac{11}{12}}^1 \psi_3(-\hat V_2(x))\,dx\\
&=\int_0^{\frac{1}{3}} \psi_3(y)\,dy+\frac{11}{12}\int_{\frac{1}{3}}^{\frac{2}{5}} \psi_3(y)\,dy
+\frac{1}{12} \int_{\frac{1}{3}}^{\frac{2}{5}} \psi_1(y)\,dy\\
&\qquad \qquad \qquad \quad+\frac{1}{2}\int_{2\over 5}^{1\over 2}\psi_3(y)\,dy+
\frac{1}{2}\int_{2\over 5}^{1\over 2}\psi_1(y)\,dy+\int_{1\over 2}^{1}\psi_1(y)\,dy.
\end{align*}
Therefore,
\[
  p_{+,\frac{1}{2}}- p_{+,2}=\frac{5}{12}\left\{ \int_{\frac{1}{3}}^{\frac{2}{5}} \psi_3(y)\,dy
-\int_{\frac{1}{3}}^{\frac{2}{5}} \psi_1(y)\,dy \right\}<0,
\]
which is absurd. 

\end{proof}

\section{Viscous Case}

For convenience,  we set the diffusive constant $d=1$ and write $\overline H_d(p)=\overline H(p)$ in this section.  

\subsection{Connection with the Hill operator}
In this subsection, we assume $n=1$.
Assume $V\in  C(\Bbb T)$.  For each $p\in \Rset$, the cell problem of interest is
\be\label{cell}
-v''+|p+v'|^2+V(x)=\overline H(p)   \quad \text{in  $\Tset$}.
\ee
It is easy to see that
\be\label{control}
 |p|^2+\int_{\Bbb T}V\,dx\leq \overline H(p)\leq  |p|^2 +\max_{\Bbb T}V.
\ee
Let us reformulate the question in the viscous case here for clarity.
\begin{quest} \label{quest-2}
For $i=1,2$, let $\overline H_i(p)$ be the viscous effective Hamiltonian associated with $V_i$.   
Assume that
$$
\overline H_1(p)=\overline H_2(p) \quad \text{for all $p\in  \Rset$}.
$$
What can we say about $V_1$ and $V_2$?
\end{quest}

For $\lambda\in   \Cset$,  let $w_1$ solve 
$$
\begin{cases}
-w_{1}''=(\lambda+V)w_1  \quad \text{in $(0,1)$},\\
w_1(0)=1,\  w_{1}'(0)=0 , 
\end{cases}
$$
and $w_2$ solve
$$
\begin{cases}
-w_{2}''=(\lambda+V)w_2  \quad \text{in $(0,1)$},\\
w_2(0)=0,\  w_{2}'(0)=1.   
\end{cases}
$$
Denote
$$
M=\left(\begin{array}{cc}
 w_1(1) & w_2(1)\\
w_{1}'(1) & w_{2}'(1)
\end{array}\right).
$$
It is clear that  $\mathrm{det}(M)=1$, and therefore, $A$ has two eigenvalues  $\theta$ and $1\over \theta$.  Denote
$$
\Delta(\lambda)=\theta+{1\over \theta}=w_1(1)+w_{2}'(1),
$$
which is the so called {\it  discriminant} associated with the Hill operator $Q=-{d^2\over dx^2}-V(x)$.  See page 295  in  \cite{RS}.  One can easily  show that   $\Delta (\lambda)$ is an entire function.  Obviously, $\Delta(\lambda)\in  \Rset$ if $\lambda\in  \Rset$. 

\begin{lem}  For $-\lambda\geq \min_{\Rset}\overline H$, 
$$
\left\{p\in  \Rset\,:\,\   \overline H(p)=-\lambda\right\}=\log \left({\Delta(\lambda)\pm \sqrt{\Delta^2 (\lambda)-4}\over 2}\right). 
$$
\end{lem}

\begin{proof}
 Let $\overline H(p)=-\lambda$ and $v$ be  a solution to cell problem  \eqref{cell}.    Denote
$$
w=e^{-(px+v)}.
$$
Then $w$ satisfies that 
$$
\begin{cases}
-w''-Vw=\lambda w  \quad \text{in  $\Rset$}\\
w(1)=e^{-p}w(0),  \  w'(1)=e^{-p}w'(0).
\end{cases}
$$
Assume that $w=a_1w_1+a_2w_2$ for $a_1, a_2\in  \Rset$.   Then  an easy calculation shows that
$$
MA=e^{-p}A,
$$
for $A=(a_1, a_2)^T$.  So  $e^{-p}$ is an eigenvalue of $M$.
\end{proof}

\medskip

Since $\Delta (\lambda)$ is an entire function,  our inverse problem (Question \ref{quest-2}) is equivalent to the following question:

\medskip
\begin{quest}\label{quest-3}
 For $i=1,2$, let $\Delta_i(\lambda)$ be the discriminant associated with $V_i$.   Assume that
$$
\Delta_1(\lambda)=\Delta_2(\lambda)  \quad \text{for all $\lambda\in  \Rset$}.
$$
What can we say about $V_1$ and $V_2$?
\end{quest}

It is known that the discriminant is determined by the spectrum of the Hill operator:  
$\Delta_1(\lambda)\equiv \Delta_2(\lambda)$ if and only  if  the following two Hill operators
$$
L_1=-{d^2\over dx^2}-V_1    \qquad   \mathrm{and}  \qquad  L_1=-{d^2\over dx^2}-V_2
$$
have  the same eigenvalues.   See Theorem XIII.92  in  \cite{RS}. 

\begin{rmk}\label{differenceviscous}
Unlike the inviscid one dimensional case,   
that $V_1$ and $V_2$ have the same distribution is neither a necessary nor a sufficient condition 
for $\Delta_1 \equiv \Delta_2$ (equivalently,  $\overline H_1=\overline H_2$.)   

To see the non-sufficiency is quite simple.   
As suggested by  Elena Kosygina,   fix $V\in C(\Bbb T)$ and look at $V_{m}=V(mx)$.   
Clearly,   $V_m$ and $V$ have the same distribution.   
However,  as $m\to +\infty$,  the viscous  effective Hamiltonian associated with $V_m$ 
converges to the inviscid  effective Hamiltonian associated with $V$.  
The invicid effective Hamiltonian is always  larger than the  the viscous one for non-constant $V$.    

The non-necessity is more tricky. 
 It is known that the KdV equation preserves the discriminant. 
More precisely,  if $q(x,t)$ is a smooth space periodic solution to the KdV equation 
\[
q_t+qq_x+q_{xxx}=0,
\]
then the spectrum (or the discriminant) associated with the Hill operator 
\[ 
L=-{d^2\over d x^2}-{1\over 6}q(\cdot,t)
\] 
is independent of $t$.   See  \cite{GGKM, Lax} for instance.   
However,   the distribution of $q(\cdot,t)$ is not invariant under the KdV equation. 
In fact, without involving derivatives of $q$, only the quantities  
$\int_{\Bbb T}q(x,t)\,dx$ and $\int_{\Bbb T}q^2(x,t)\,dx$ are conserved.  
Hence (3) in  Theorem \ref{m7} is optimal. 
\end{rmk}

\subsection{Proof of Theorem \ref{m6}}    
\begin{proof}   We first prove (1).   Owing to part (1) of  Theorem  \ref{m7},  we have that
$$
\int_{\Bbb T^n}V\,dx=0.
$$
Since $H$ is superlinear, we may choose $p_0\in  \Rset^n$ such that 
$$
H(p_0)-\sqrt {1+|p_0|^2} =\min_{\Rset^n}\left (H(p)-\sqrt {1+|p|^2} \right )=h_0.
$$
Denote  $\tilde H(p)=\sqrt {1+|p|^2} +h_0$.   Then $H(p)\geq \tilde H(p)$ and $H(p_0)=\tilde H (p_0)$.  Clearly,
$$
H=\overline H\geq \overline {\tilde H}\geq  \tilde H +\int_{\Bbb T^n}V\,dx=\tilde H.
$$
Here $\overline {\tilde H}$  represents the viscous effective Hamiltonian associated with $\tilde H+V$ with $d=1$.   Accordingly,
$$
\overline {\tilde H}(p_0)=\tilde H(p_0)+\int_{\Bbb T^n}V\,dx=\tilde H(p_0).
$$
Let $v_0\in  C^{\infty}(\Bbb T^n)$ be a solution to 
$$
-\Delta  v_0+\tilde H(p_0+Dv_0)+V=\overline {\tilde H}(p_0)=\tilde H(p_0) \quad \text{in  $\Tset^n$}.
$$
Taking integration over $\Bbb T^n$ and using the strict convexity of $\tilde H$,  we obtain that $Dv_0\equiv 0$.   Hence $V\equiv 0$.

\medskip

Next we prove  (2).  Let $v=v(x,p)\in C^{\infty}(\Tset^n)$ be the unique solution to 
$$
\begin{cases}
-\Delta v+|p+Dv|^2+V(x)=\overline H(p)    \quad \text{in  $\Tset ^n$},\\
 \int_{\Bbb T^n}v\,dx=0.  
\end{cases}
$$
Then  $w=e^{-v}$ satisfies that 
$$
\Delta w-2p\cdot Dw+V(x)w=(\overline H(p)-|p|^2)w    \quad \text{in  $\Tset ^n$}.
$$
For $w_0(x)=w(x,0)$, it is clear that
\begin{equation}\label{eqn-w0}
\Delta w_0+V(x)w_0=0  \quad \text{in} \ \Tset^n.
\end{equation}
Taking partial derivatives of $w(x,p)$ with respect to $p_1$ and evaluating at $p=0$,  we obtain that,
for  $w_1(x)=w_{p_1}(x,0)$ and $w_2(x)=w_{p_1p_1}(x,0)$, 
\begin{equation}\label{eqn-w1}
 \Delta w_1+V(x)w_1=2w_{x_1}  \quad \text{in} \ \Tset^n
\end{equation}
and
\begin{equation}\label{eqn-w2}
\Delta w_2+V(x)w_2=4w_{x_1p_1}   \quad \text{in} \ \Tset^n. 
\end{equation}
 Accordingly,  for $Q=[0,1]^n$,
$$
\int_{Q}w_0w_{x_1p_1}(x,0)\,dx=0.
$$
This implies that $\int_{Q}w_1w_{x_1}(x,0)\,dx=0$.   In light of \eqref{eqn-w1}, one deduces that
$$
\int_{Q}\left(|Dw_1|^2-V(x)w_{1}^{2}\right)\,dx=0.
$$
Since $w_0>0$ is the principle eigenfunction of the symmetric operator $L=-\Delta-V$,  we have that
$$
0=\min_{\phi\in  H^1(\Bbb T^n)}\int_{Q}\left(|D\phi|^2-V(x)\phi^2\right)\,dx.
$$
Hence $w_1=\lambda w_0$  for some $\lambda\in  \Rset$, which gives that $w_{x_1}\equiv 0$.   
Similarly,  we can show that $w_{x_i}(x,0) \equiv 0$ for $i\geq 2$.  

Therefore, $w_0$ is a positive constant and $V\equiv 0$.  

\end{proof}

\subsection{Proof of Theorem \ref{m7}} 
\begin{proof}

The proof of  \eqref{difference-viscous}   is essentially the same as \eqref{difference1}. 
Since $Q$ satifies a Diophantine condition,  there exists a unique smooth periodic  solution $v$ (up to an additive constant)  to
$$
 Q\cdot Dv=a_1-V_1 \quad \text{in} \ \Tset^n,
$$
for  $a_1=\int_{\Bbb T^n}V_1\,dx$.  Then it is easy to see that for $v_{\lambda}={v\over \lambda}$, 
$$
-\Delta v_{\lambda}+H(\lambda P_\lambda+Dv_{\lambda})+V(x)=H(\lambda P_{\lambda})+a_1+O\left({1\over \lambda}\right) \quad \text{in  $\Tset^n$}.
$$
Let $w_{\lambda}\in C^{\infty}(\Bbb T^n)$ be a viscosity solution to 
$$
-\Delta w_\lambda+H(\lambda P_\lambda+Dw_{\lambda})+V(x)=\overline H_1(\lambda P_{\lambda})  \quad \text{in  $\Tset^n$}.
$$
By looking at the places where $w_{\lambda}-v_{\lambda}$ attains its maximum and minimum,  we get that
$$
\overline H_1(\lambda P_{\lambda})=H(\lambda P_{\lambda})+a_1+O\left({1\over \lambda}\right).
$$

\medskip

Next we prove (2).   Let  $Q$ be a unit vector satisfying a  Diophantine condition.  For $i=1,2$,   we can explicitly solve the following equations  in $\Tset^n$ by computing Fourier coefficients
\be\label{eqgroup}
\begin{cases}
2Q\cdot Dv_{i1}=a_{i1}-V_i\\
2Q\cdot Dv_{i2}=\Delta v_{i1}\\
2Q\cdot Dv_{i3}=a_{i2}-|Dv_{i1}|^2+\Delta v_{i2}.
\end{cases}
\ee
Here $a_{i1}=\int_{\Bbb T^n}V_i\,dx$ and $a_{i2}=\int_{\Bbb T^n}|Dv_1|^2\,dx.$  Then for $\ep>0$, $v_{i\ep}=\ep^2 v_{i1}+ \ep ^3 v_{i2}+ \ep ^4 v_{i3}$ satisfy
$$
-\ep\Delta v_{i\ep}+|Q+Dv_{i\ep}|^2+\ep^2 V_i=|Q|^2+\ep a_{i1} + \ep^4 a_{i2}+O(\ep ^5).
$$
Suppose that $w_{i\ep}\in  C^{\infty}(\Bbb T^n)$ is a  solution to 
$$
-\ep\Delta w_{i\ep}+|Q+Dw_{i\ep}|^2+\ep^2 V_i=\overline H_{i\ep}(Q)  \quad \text{in $\Rset$}.
$$
Here $\overline H_{i\ep}(Q)=\ep ^2\overline H_i({Q\over \ep})$ is the viscous effective Hamiltonian associated with $|p|^2+\ep V_i$  with $d=\ep$.  
By looking at places where $v_{i\ep}-w_{i\ep}$ attains its maximum and minimum, we derive that
$$
\overline H_{i\ep}(Q)=|Q|^2+\ep^2 a_{i1} + \ep^4 a_{i2}+O(\ep ^5).
$$
To finish the proof,   it suffices to show that for any $Q$ satisfying a  Diophantine condition
$$
\int_{\Bbb T^n}|Dv_{11}|^2\,dx=\int_{\Bbb T^n}|D v_{21}|^2\,dx, 
$$
then
$$
\int_{\Bbb T^n}|V_{1}|^{2}\,dx=\int_{\Bbb T^n}|V_{2}|^{2}\,dx.
$$
From here, the proof goes exactly the same as that of (3) in   Theorem \ref{m2}.

\medskip

Finally,  (3) follows immediately from the fact that
$$
v_{i1}'={1\over 2}(a_{i1}-V_i) \quad \text{in} \ \Tset,
$$
for $n=1$,  $Q=1$ and $i=1,2$.
\end{proof}

\bibliographystyle{plain}

\end{document}